\documentclass[a4paper,12pt]{article}

\usepackage{amscd}
\usepackage{amssymb,amsfonts,amsmath,amsthm}
\usepackage{verbatim}

\usepackage{amsmath}
\usepackage{amsthm}
\usepackage{amssymb}

\usepackage{latexsym}
\usepackage{array}
\usepackage{enumerate}

\usepackage[all]{xy}
\usepackage{lipsum}

\numberwithin{equation}{section}

\newcommand{\BDC}{{\mathbf{D}}^{\mathrm{b}}}

\newcommand{\Hom}{\mathrm{Hom}}

\newcommand{\CC}{\mathbb{C}}

\newcommand{\RR}{\mathbb{R}}
\newcommand{\QQ}{\mathbb{Q}}
\newcommand{\ZZ}{\mathbb{Z}}

\newcommand{\E}{\mathcal{E}}
\newcommand{\F}{\mathcal{F}}

\newcommand{\SL}{\mathcal{L}}
\newcommand{\M}{\mathcal{M}}

\newcommand{\PP}{{\mathbb P}}

\newcommand{\KK}{{\rm K}}

\newcommand{\HSm}{{\rm HS}^{\rm mon}}
\newcommand{\LL}{{\mathbb L}}

\newcommand{\Var}{{\rm Var}}

\renewcommand{\dim}{{\rm dim}}

\newcommand{\Vol}{{\rm Vol}}

\newcommand{\e}{\varepsilon}
\newcommand{\Spec}{{\rm Spec}}
\newcommand{\Aff}{{\rm Aff}}

\newcommand{\Int}{{\rm Int}}
\newcommand{\relint}{{\rm rel.int}}

\newcommand{\tl}[1]{\widetilde{#1}}
\newcommand{\ov}[1]{\overline{#1}}

\newcommand{\simto}{\overset{\sim}{\longrightarrow}}

\newcommand{\dsum}{\displaystyle \sum}

\newcommand{\BK}{\mathbb{K}}

\newcommand{\CS}{\mathbb{C}^{*}}
\DeclareMathOperator{\Conv}{Conv}
\DeclareMathOperator{\ord}{ord}
\DeclareMathOperator{\INT}{int}

\newcommand{\MH}{\hat{\mu}}

\newcommand{\UH}{\mathrm{UH}}
\newcommand{\LK}{\mathrm{lk}}
\newcommand{\MCS}{\mathcal{S}}

\newtheorem{theorem}{Theorem}[section]

\newtheorem{corollary}[theorem]{Corollary}
\newtheorem{lemma}[theorem]{Lemma}
\newtheorem{proposition}[theorem]{Proposition}

\theoremstyle{definition}
\newtheorem{definition}[theorem]{Definition}
\theoremstyle{remark}
\newtheorem{remark}[theorem]{\sc Remark}
\title{
On the monodromies and the limit mixed Hodge 
structures of families of algebraic varieties 
}

\author{
Takahiro SAITO 
\footnote{Research Institute for Mathematical 
Sciences, Kyoto University, Kyoto 606-8502, Japan.
E-mail: takahiro@kurims.kyoto-u.ac.jp}
and Kiyoshi TAKEUCHI 
\footnote{Mathematical Institute, Tohoku University,
Aramaki Aza-Aoba 6-3, Aobaku, Sendai, 980-8578, Japan.
E-mail: takemicro@nifty.com } }

\date{}

\begin{document}

\maketitle
\begin{abstract}
We study the monodromies and the limit mixed Hodge structures 
of families of complete intersection 
varieties over a punctured disk in 
the complex plane. For this purpose, we express their 
motivic nearby fibers in terms of the geometric data of 
some Newton polyhedra. In particular, the limit mixed Hodge 
numbers and some part of the Jordan normal forms of 
the monodromies of such a family will be described 
very explicitly. 
\end{abstract}

\maketitle

\section{Introduction}

\label{sec:1}
Families of algebraic varieties are basic objects in algebraic 
geometry. Here we are interested in the special but 
fundamental case where such one $Y$ is smooth and defined 
over the punctured disk $B(0; \e )^*= 
\{ t \in \CC \ | \ 0<|t|< \e \} \subset \CC$ 
($0< \e \ll 1$) in $\CC$. For this family $\pi_Y: Y 
\longrightarrow B(0; \e )^*$ let us consider 
its fibers $Y_t= \pi_Y^{-1}(t) \subset Y$ 
($0<|t|< \e$) and their cohomology groups 
$H^j(Y_t; \CC )$ ($j \in \ZZ$). Then it 
is our primary interest to know $H^j(Y_t; \CC )$ 
themselves and the monodromy operators acting on them. 
Moreover, we have
the limit mixed Hodge structure 
$H^j(Y_{\infty}; \CC )$
which encodes some information of the monodromy
(see El Zein \cite{E-Z} and
Steenbrink-Zucker \cite{S-Z}).
However in general, it is very hard to compute 
the monodromies and the limit mixed Hodge numbers explicitly. 
Very recently, based on our previous works 
\cite{E-T}, \cite{M-T-4}, \cite{M-T-5} and 
some new results in \cite{K-S-1}, \cite{K-S-2}, 
Stapledon \cite{Stapledon} succeeded in computing 
the latter ones for families 
$Y \subset B(0; \e )^* \times \CC^n$ 
of sch\"{o}n hypersurfaces in 
$\CC^n$. Here the sch\"{o}nness is a very weak 
condition which is almost always satisfied. 
More precisely, in \cite{Stapledon} the author 
expressed the motivic nearby fiber $\psi_t ([Y])$ of 
such $Y \longrightarrow 
B(0; \e )= \{ t \in \CC \ | \ |t|< \e \}
=B(0; \e )^* \sqcup \{ 0 \}$ by 
the function $t={\rm id}_{\CC}: \CC \longrightarrow \CC$ 
in terms of the tropical variety associated to 
the defining Laurent polynomial $f(t,x) \in \CC(t) 
[x_{1},\dots, x_{n}]$ 
of the hypersurface $Y \subset 
B(0; \e )^* \times \CC^n$ and obtained a 
complete description of the limit mixed Hodge numbers 
of $H^j(Y_\infty; \CC )$. 
His idea is to subdivide $\CC^n$ into 
some algebraic tori $( \CC^*)^k$ 
($0 \leq k \leq n$) by the additivity of 
$\psi_t$ and apply the arguments of 
Batyrev-Borisov \cite{B-B} and 
Borisov-Mavlyutov \cite{B-M} in mirror 
symmetry to each piece $( \CC^*)^k$ by using 
the new special polynomials that
provide a common generalization for polynomials introduced in 
Katz-Stapledon \cite{K-S-1} and Stapledon~\cite{StapWeight}. In particular, he 
effectively used the purity 
and the generalized Poincar\'{e} duality for the 
intersection cohomology groups of some 
singular hypersurfaces in the toric 
compactifications of $( \CC^*)^k$ to obtain 
these remarkable results. 
Thanks to them, the first author T. Saito 
obtained some new results on the weight 
filtrations of the stalks of intersection 
cohomology complexes. See \cite{Saito} for 
the details. 
However the motivic nearby 
fiber $\psi_t ([Y])$ used in the paper 
\cite{Stapledon} is the one introduced by Steenbrink 
\cite{Steenbrink} very recently with the help of 
the semi-stable reduction theorem and 
it is not clear for us if it coincides with the 
classical (and more standard) one of 
\cite{D-L-1}, \cite{D-L-2} and \cite{G-L-M} (see also Raibaut \cite{R-2} 
for a nice introduction to this subject). 
Moreover the arguments in \cite{Stapledon} 
heavily depend on some deep technical results in 
recent tropical geometry. It is therefore desirable 
to use the classical motivic nearby fibers and 
describe them without using the tropical geometry. 
The aim of this paper is to simplify Stapledon's 
arguments and extend his results to 
families of sch\"{o}n complete intersection subvarieties 
in $\CC^n$. Moreover, by extending our previous results in 
\cite{E-T}, \cite{M-T-4}, \cite{T-T}, as in \cite{T-T} 
we show that some parts of the Jordan normal forms of 
the monodromies on $H^j(Y_t; \CC )$ ($0<|t|< \e$) 
can be described very explicitly. 

In order to explain these results more precisely, 
let us introduce our geometric situation 
and notations. Let $\BK=\CC(t)$ be the field 
of rational functions of $t$ and
$f(t,x)=\sum_{v \in \ZZ^n}a_{v}(t)x^{v} 
\in\BK [x_{1}^{\pm},\dots,x_{n}^{\pm}] 
\ (a_{v}(t)\in\BK)$
a Laurent polynomial of $x=(x_{1},\dots,x_{n})$ 
with coefficients in $\BK$. Then we define a 
family $Y \subset B(0; \e )^* \times ( \CC^*)^n$ 
of hypersurfaces in the algebraic torus 
$(\CC^*)^n$ over the punctured disk $B(0; \e )^*$ by 
$Y=f^{-1}(0) \subset 
B(0; \e )^* \times (\CC^*)^n$. 
If $f$ has a very special form $f(t,x)=1-tg(x)$ 
for some Laurent polynomial 
$g(x) \in \CC [x_{1}^{\pm},\dots,x_{n}^{\pm}]$ 
the monodromy of $Y$ around the origin 
$0 \in \CC$ is nothing but the monodromy at 
infinity of the polynomial map 
$g: ( \CC^*)^n \longrightarrow \CC$. In this 
sense, our setting is a vast generalization 
of the classical ones of 
\cite{Broughton}, \cite{E-T}, \cite{L-N-2}, 
\cite{L-S}, \cite{M-T-2}, \cite{M-T-4}, 
\cite{M-T-6}, \cite{Raibaut}, \cite{Sabbah-1}, 
\cite{Sabbah-2}, \cite{S-T-1}, \cite{Tibar}. 
For $v\in\ZZ^n$ by the Laurent 
expansion $a_{v}(t)=\sum_{j\in\ZZ}a_{v,j}
t^j\ (a_{v,j}\in\CC)$ of the rational function $a_{v}(t)$ 
we set
\[ o(v):=\ord_{t}a_{v}(t)=\min\{j \ | \ a_{v,j}\neq 0\}.\]
If $a_{v}(t) \equiv 0$ we set $o(v)=+\infty$.
Then we define an (unbounded) polyhedron 
$\UH_{f}$ in $\RR^{n+1}$ by
\[\UH_{f} = \Conv{\Bigl[ \bigcup_{v\in\ZZ^n}\{(v,s) 
 \in \RR^{n+1} \ | \ 
s \geq o(v)\}  
\Bigr]} \subset \RR^{n+1},\]
where $\Conv{(\ \cdot\ )}$ stands for the convex hull. 
We call it the upper-half polyhedron of $f$. 
Throughout this paper we assume that the 
dimension of $\UH_{f}$ is $n+1$.
Then by the projection $p:\RR^{n+1}
=\RR^n\times\RR^{1}\twoheadrightarrow 
\RR^n$ we obtain an $n$-dimensional polytope 
$P:=p(\UH_{f}) \subset \RR^n$ 
which we call the Newton polytope of $f$. 
Let $\nu_{f}:P\to \RR$ be the function 
defining the bottom part of the boundary 
$\partial \UH_{f}$ of $\UH_{f}$ and 
$\mathcal{S}$ the subdivision of $P$ by the lattice polytopes
$p(\widetilde{F})\subset \RR^n\ (\widetilde{F}\prec \UH_{f})$. 
We call such $F= p(\widetilde{F}) \in \mathcal{S}$ 
a cell in $\mathcal{S}$. 
We denote by $\relint F$ its relative interior 
i.e. its interior in the affine span 
${\rm Aff}(F) \simeq \RR^{\dim F}$ of $F$. 
Then by defining a hypersurface 
$V_F$ of the algebraic torus
$T_{F}=\Spec(\CC[{\rm Aff}(F) 
\cap\ZZ^n]) \simeq (\CC^*)^{\dim F}$ 
for each cell $F \in \mathcal{S}$ and the 
sch\"{o}nness of the family $Y$ etc. 
(see Section \ref{sec:3} for the details), 
we reobtain the following beautiful 
result of Stapledon \cite{Stapledon}. 
For the family $Y$ over the punctured disk, 
denote by $\psi_{t}([Y]) \in 
\M_{\CC}^{\hat{\mu}}$ 
its motivic nearby fiber by the function
$t= {\rm id}_{\CC}:\CC \longrightarrow \CC$. 
Here $\M_{\CC}^{\hat{\mu}}$ stands for 
a localization of the Grothendieck group 
of varieties with good actions of the 
group $\hat{\mu}  = \underset{d 
\in \ZZ_{>0}}{\varprojlim} \ZZ/\ZZ d$ 
(see Section \ref{sec:2} for the details).

\begin{theorem}\label{RCS}
Assume that the family $Y$ of 
hypersurfaces in $( \CC^*)^n$ is sch\"on. 
Then we have an equality
\[\psi_{t}([Y])=\sum_{\relint F \subset 
\Int P}[V_{F}\circlearrowleft \MH]\cdot 
(1-\mathbb{L})^{n-\dim{F}}\]
in $\M_{\CC}^{\hat{\mu}}$, 
where $[V_{F}\circlearrowleft \MH] \in \M_{\CC}^{\hat{\mu}}$ 
(resp. $\mathbb{L} \in \M_{\CC}^{\hat{\mu}}$) 
is the element of $\M_{\CC}^{\hat{\mu}}$ 
defined by $V_F$ (resp. the complex line $\CC$) 
endowed with a natural (resp. the trivial) 
action of $\hat{\mu}$ (see Section \ref{sec:3} for the details).
\end{theorem}
We prove this theorem by using only the classical 
toric geometry and a result of 
Guibert-Loeser-Merle \cite{G-L-M} 
(see Theorem \ref{GLMT}). 
Then in Section \ref{sec:4} we apply it to families 
$Y \subset B(0; \e )^* \times \CC^n$ 
of sch\"{o}n hypersurfaces   
in $\CC^n$ and describe 
some parts of the Jordan normal forms of 
the monodromies on $H^j(Y_t; \CC )$ ($0<|t|< \e$) explicitly. 
More precisely, as in \cite{T-T} we define 
a finite subset $R_f$ of $\CC$ by $\UH_{f}$ and 
describe the Jordan normal forms for 
the eigenvalues $\lambda \notin R_f$. 
For this purpose, first we prove the following 
concentration theorem. For $j \in \ZZ$ and 
$\lambda \in \CC$ let 
\[H^{j}(Y_{t};\CC)_{\lambda} \subset H^{j}(Y_{t};\CC)
\qquad ({\rm resp.} \ 
H^{j}_{c}(Y_{t};\CC)_{\lambda} \subset H^{j}_{c}(Y_{t};\CC) 
) \]
be the generalized eigenspace of the monodromy automorphism 
$\Psi_{j}:H^{j}(Y_{t};\CC)\xrightarrow{\sim}
H^{j}(Y_{t};\CC)$ (resp. 
$\Phi_{j}:H^{j}_{c}(Y_{t};\CC)\xrightarrow{\sim}
H^{j}_{c}(Y_{t};\CC)$) for $0<|t|<\varepsilon$.

\begin{theorem}\label{BCTM}
(see Theorem \ref{th:5}, 
Corollary \ref{Ncc} and Remark \ref{imrem}) 
Assume that the family $Y$ of hypersurfaces 
in $\CC^{n}$ is sch\"on. 
Then for any $\lambda\notin R_{f}$ 
and $t\in \CS$ such that $0<|t|\ll1$ we have 
isomorphisms 
\[H^{j}_{c}(Y_{t};\CC)_{\lambda} 
\simeq H^{j}(Y_{t};\CC)_{\lambda} \qquad (j \in \ZZ )\]
and the concentration 
\[H^{j}(Y_{t};\CC)_{\lambda} 
\simeq 0 \qquad (j \not= n-1).\]
\end{theorem}
By the proofs of this theorem (see Theorem \ref{th:5}) 
and Sabbah's one \cite[Theorem 13.1]{Sabbah-2} 
we see also that for $\lambda\notin R_{f}$ 
the filtration on the only non-trivial 
cohomology group $H^{n-1}(Y_{t};\CC)_{\lambda}$ 
induced by Deligne's weight filtration on 
$H^{n-1}(Y_{t};\CC)$ is concentrated in 
degree $n-1$. Since $R_f$ is just a small part 
of the set of the eigenvalues of the monodromies 
of $Y$, Theorem \ref{BCTM} asserts that 
the geometric complexity of the family $Y$ is 
concentrated in the middle dimension 
$n-1= \dim Y_t$. This enables us to describe 
the Jordan normal forms of the middle-dimensional 
monodromies 
\[ 
\Psi_{n-1}:H^{n-1}(Y_{t};\CC)_{\lambda} \xrightarrow{\sim}
H^{n-1}(Y_{t};\CC)_{\lambda} \qquad (0<|t|< \e ) 
\]
for the eigenvalues $\lambda\notin R_{f}$ as follows,  
in terms of the equivariant limit mixed Hodge polynomials 
obtained by Stapledon's results in \cite{Stapledon}. 

\begin{theorem}\label{NJBTH}
Assume that the family $Y$ of hypersurfaces 
in $\CC^{n}$ is sch\"on. 
For $\lambda\in \CC$ and $m\geq 1$ denote by $J_{\lambda,m}$ the 
number of the Jordan blocks in the monodromy 
\[ 
\Psi_{n-1}:H^{n-1}(Y_{t};\CC)_{\lambda} \xrightarrow{\sim}
H^{n-1}(Y_{t};\CC)_{\lambda} \qquad (0<|t|< \e ) 
\]
for the eigenvalue $\lambda$ with size $m$.
Then for $\lambda\notin R_{f}$ we have
\[\sum_{m=0}^{n-1}J_{\lambda,n-m}s^{m+2}=\sum_{F\in\mathcal{S}}
s^{\dim{F}+1}l^{*}_{\lambda}(F,\nu_{f}|_{F};1) 
  \cdot \tl{l}_{P}(
\mathcal{S},F;s^{2})\]
(for the definitions of the polynomials 
$l^{*}_{\lambda}(F,\nu_{f}|_{F};u) \in \ZZ [u]$ and 
$\tl{l}_{P}( \mathcal{S},F; t) \in \ZZ [t]$ 
see Sections \ref{sec:2} and \ref{sec:3}). 
\end{theorem}
The proof of Theorem \ref{BCTM} relies on 
the decompositions of nearby cycle 
perverse sheaves associated to normal crossing 
divisors and we apply it on some smooth 
toric varieties. 
For the proof, we are also indebted to 
\cite[Theorem 5.7]{Stapledon} which is 
proved by using \cite[Section 2]{M-T-4} and 
some deep results on combinatorics 
developed by Katz-Stapledon in \cite{K-S-1} 
that build on earlier work of Stanley~\cite{Stanley}. 
Theorem \ref{BCTM} holds true 
without any assumption on the shape of 
the Newton polytope $P=p(\UH_{f}) \subset \RR^n$. 
In particular, we do not require here that $P$ 
is convenient as in \cite{Stapledon}. The convenience of $P$ 
of \cite{Stapledon} is stronger than the usual one 
and we cannot expect it in general (see Remark~\ref{StaCom}). For the treatment of 
the non-convenient case, we have to prove the 
topological concentration 
\[H^{j}(Y_{t};\CC)_{\lambda} 
\simeq 0 \qquad ( \lambda\notin R_{f}, j \not= n-1)\]
in Theorem \ref{BCTM} which does not 
follow from the results in 
Danilov-Khovanskii \cite{D-K} and 
Stapledon \cite{Stapledon}. 
Moreover in Section \ref{sec:5}, we also extend 
these results to 
families of sch\"{o}n complete intersection subvarieties 
in $\CC^n$ and obtain a formula for the 
Jordan normal forms of their monodromies. To our surprise, 
the results that we obtain in this generalized 
situation are completely parallel to the ones 
for families of hypersurfaces in $\CC^n$. 
See Section \ref{sec:5} for the details. 
\\\\\noindent{\bf Acknowledgement:}
The authors thank 
Professors Masaharu Ishikawa and 
Sampei Usui for their encouragement. 
They are grateful also to the anonymous referee 
whose comments improved their paper substantially.

\section{Preliminary notions and results}\label{sec:2}

In this section, we introduce some preliminary notions and results 
which will be used in this paper. 

\subsection{Motivic nearby fibers}
Throughout this paper we consider only varieties 
over the field $\CC$ of complex numbers. 
From now we shall introduce the theory of motivic 
nearby fibers of Denef-Loeser \cite{D-L-1}, 
\cite{D-L-2} and Guibert-Loeser-Merle \cite{G-L-M} 
in this special case (see also Raibaut \cite{R-2}). 
For a variety $S$ denote by $\KK_0(\Var_{S})$ 
the  Grothendieck ring of varieties over $S$. 
Recall that the ring structure is defined by 
the fiber products over $S$. 
Moreover we denote by $\M_{S}$ the ring obtained 
from it by inverting the 
Lefschetz motive $\LL\simeq \CC \times S \in 
\KK_0(\Var_{S})$. If $S= \Spec ( \CC )$ we denote 
$\KK_0(\Var_{S})$ and $\M_{S}$ simply by 
$\KK_0(\Var_{\CC})$ and $\M_{\CC}$ respectively. 
Note that $\M_S$ has a natural structure of an 
$\M_{\CC}$-module. For $d \in \ZZ_{>0}$, let 
$\mu_d = \{ \zeta \in \CC \ | \ 
\zeta^d=1 \}  \simeq \ZZ/\ZZ d$ 
be the multiplicative 
group consisting of $d$ roots of unity 
in $\CC$. We denote by $\hat{\mu}$ the 
projective limit $\underset{d}{\varprojlim} 
\mu_d$ of the projective system 
$\{ \mu_i \}_{i \geq 1}$ with morphisms 
$\mu_{id} \longrightarrow \mu_i$ 
given by $t \longmapsto t^d$. 
Then we define the Grothendieck ring 
$\KK_0^{\hat{\mu}}(\Var_{S})$ of 
varieties over $S$ with good 
$\hat{\mu}$-actions and its localization 
$\M_{S}^{\hat{\mu}}$ 
as in \cite[Section 2.4]{D-L-2}. 
Note that $\M_{S}^{\hat{\mu}}$ is naturally 
a $\M_{\CC}$-module. Recall also 
that for a morphism 
$\pi : S \longrightarrow S^{\prime}$ 
of varieties we have a group homomorphism 
\begin{equation}
\pi_!: 
\M_{S}^{\hat{\mu}} \longrightarrow 
\M_{S^{\prime}}^{\hat{\mu}}
\end{equation}
obtained by the composition with $\pi$. 
Now let $Z$ be a smooth variety and 
$U \subset Z$ a Zariski open subset 
such that $D=Z \setminus U$ is a 
normal crossing divisor in $Z$. 
Moreover let $f:Z \longrightarrow \CC$ 
be a regular function on $Z$ such that 
$f^{-1}(0) \subset Z$ is contained in $D$. 
Denote by $D^{\prime} \subset D$ the 
union of irreducible components of $D$ 
which are not contained in $f^{-1}(0)$ 
and set $\Omega = Z \setminus D^{\prime}$. 
Then we have $U \subset \Omega$. 
Let $E_1, E_2, \ldots, E_k$ 
be the irreducible components of 
the normal crossing divisor $\Omega 
\cap f^{-1}(0)$ in $\Omega \subset Z$. 
For each $1 \leq i \leq k$, let $b_i>0$ be 
the order of the zero of $f$ 
along $E_i$. For a non-empty subset $I 
\subset \{1,2,\ldots, k\}$, let us 
set
\begin{equation}
E_I=\bigcap_{i \in I} E_i,\quad
E_I^{\circ}=E_I \setminus \bigcup_{i 
\not\in I}E_i
\end{equation}
and $d_I= {\rm gcd} (b_i)_{i \in I}>0$. Then, 
as in \cite[Section 3.3]{D-L-2}, we 
can construct an unramified Galois 
covering $\tl{E_I^{\circ}} 
\longrightarrow E_I^{\circ}$ of $E_I^{\circ}$ 
as follows. First, for a point 
$p \in E_I^{\circ}$ we take an affine open 
neighborhood $W \subset \Omega 
\setminus ( \cup_{i \notin I} E_i)$ of $p$ 
on which there exist regular functions 
$\xi_i$ $(i \in I)$ 
such that $E_i \cap W=\{ \xi_i=0 \}$ 
for any $i \in I$. Then on $W$ we have 
$f=f_{1,W}(f_{2,W})^{d_I}$, 
where we set 
$f_{1,W}=f \prod_{i \in I}\xi_i^{-b_i}$ 
and $f_{2,W}=\prod_{i 
\in I} \xi_i^{\frac{b_i}{d_I}}$. 
Note that $f_{1,W}$ is a unit on $W$ 
and $f_{2,W} \colon W \longrightarrow 
\CC$ is a regular function. It is 
easy to see that $E_I^{\circ}$ is covered 
by such affine open subsets $W$ of 
$\Omega \setminus ( \cup_{i \notin I} E_i)$. 
Then as in \cite[Section 3.3]{D-L-2} 
by gluing the varieties
\begin{equation}\label{eq:6-26}
\tl{E_{I,W}^{\circ}}=\{(t,z) \in 
\CC^* \times (E_I^{\circ} \cap W) \ |\ 
t^{d_I} =(f_{1,W})^{-1}(z)\}
\end{equation}
together in an obvious way, 
we obtain the variety $\tl{E_I^{\circ}}$ over 
$E_I^{\circ}$. This 
unramified Galois covering 
$\tl{E_I^{\circ}}$ of $E_I^{\circ}$ 
admits a natural $\mu_{d_I}$-action 
defined by assigning the automorphism 
$(t,z) \longmapsto (\zeta_{d_I} t, z)$ 
of $\tl{E_I^{\circ}}$ to the generator 
$\zeta_{d_I}:=\exp 
(2\pi\sqrt{-1}/d_I) \in \mu_{d_I}$. Namely 
the variety $\tl{E_I^{\circ}}$ is 
endowed with a good $\hat{\mu}$-action 
in the sense of \cite[Section 
2.4]{D-L-2} and defines an element 
$[\tl{E_I^{\circ}}]$ of 
$\M_{f^{-1}(0)}^{\hat{\mu}}$. Finally 
we set 
\begin{equation}\label{MNFF}
\mathcal{S}_{f,U}= \sum_{I \neq \emptyset} 
(1-\LL)^{|I| -1} \cdot 
[\tl{E_I^{\circ}}] \in 
\M_{f^{-1}(0)}^{\hat{\mu}}.
\end{equation}
Then we have the following result. 

\begin{theorem}[{\cite[Theorem 3.9]{G-L-M}}]\label{GLMT} 
Let $X$ be a variety and $g:X \longrightarrow \CC$ 
a regular function on it. Then 
there exists a morphism of $\M_{\CC}$-modules 
\begin{equation}
\psi_g: \M_X \longrightarrow 
\M_{g^{-1}(0)}^{\hat{\mu}}
\end{equation}
such that for 
any proper morphism $\pi : Z \longrightarrow X$ 
from a smooth variety $Z$ and a Zariski open subset 
$U \subset Z$ whose complement $D=Z \setminus U$ is a 
normal crossing divisor in $Z$ containing 
$(g \circ \pi )^{-1}(0)$ we have the equality 
\begin{equation}
\psi_g ([U \longrightarrow X])= 
(\pi|_{(g \circ \pi )^{-1}(0)})_! 
(\mathcal{S}_{g \circ \pi,U})
\end{equation}
in $\M_{g^{-1}(0)}^{\hat{\mu}}$. 
\end{theorem}

\begin{definition}
In the situation of Theorem \ref{GLMT}, 
for a variety $[Y \longrightarrow X] \in \M_X$ 
over $X$ we call $\psi_g ([Y \longrightarrow X]) \in 
\M_{g^{-1}(0)}^{\hat{\mu}}$ the motivic nearby 
fiber of $Y$ by $g$ and denote it simply by 
$\psi_g ([Y])$. 
\end{definition}

Following the notations in 
\cite[Sections 3.1.2 and 3.1.3]{D-L-2}, 
we denote by $\HSm$ the abelian 
category of Hodge structures with a 
quasi-unipotent endomorphism.  Let
\begin{equation}
\chi_h \colon \M_{\CC}^{\hat{\mu}} 
\longrightarrow \KK_0(\HSm)
\end{equation}
be the Hodge characteristic morphism 
defined in \cite{D-L-2} which 
associates to a variety $Z$ with a 
good $\mu_d$-action the Hodge structure
\begin{equation}
\chi_h ([Z])=\sum_{j \in \ZZ} (-1)^j 
[H_c^j(Z;\QQ)] \in \KK_0(\HSm)
\end{equation}
with the actions induced by the one 
$z \longmapsto \exp (2\pi\sqrt{-1}/d)z$ 
($z\in Z$) on $Z$. We can generalize this 
construction as follows. For a variety $X$ 
let ${\rm MHM}_X$ be the abelian category of 
mixed Hodge modules on $X$ (see \cite{MHM}) and 
$\KK_0({\rm MHM}_X)$ its Grothendieck ring. 
Then there exists a group homomorphism 
\begin{equation}
H_X \colon \M_X 
\longrightarrow \KK_0({\rm MHM}_X)
\end{equation}
such that for 
any morphism $\pi \colon Z \longrightarrow X$ 
from a smooth variety $Z$ and 
the trivial Hodge module $\QQ^{H}_Z$ on it 
we have 
\begin{equation}
H_X ([Z \longrightarrow X])=
\sum_{j \in \ZZ} (-1)^j [H^jR \pi_!( \QQ^{H}_Z)]. 
\end{equation}
Here the Grothendieck ring 
$\KK_0({\rm MHM}_X)$ has a natural $\M_{\CC}$-module 
structure defined by the Hodge realization map 
$H\colon \M_{\CC} \longrightarrow 
\KK_0({\rm MHM}_{\Spec ( \CC )})$ and $H_X$ is moreover 
$\M_{\CC}$-linear. By using the abelian category 
${\rm MHM}_X^{\rm mon}$ of 
mixed Hodge modules on $X$ with a finite order 
automorphism and its Grothendieck ring 
$\KK_0({\rm MHM}_X^{\rm mon})$ we have also a 
group homomorphism 
\begin{equation}
H_X^{\rm mon} \colon \M_X^{\hat{\mu}}  
\longrightarrow \KK_0({\rm MHM}_X^{\rm mon})
\end{equation}
(see \cite{G-L-M} and \cite{R-2} for the details). 

\begin{proposition}[{\cite[Proposition 3.17]{G-L-M}}]\label{GLMCD} 
In the situation of Theorem \ref{GLMT} 
there exists a commutative diagram 
\begin{equation}
\begin{CD}
\M_X  @>{\psi_g}>> \M_{g^{-1}(0)}^{\hat{\mu}} 
\\
@V{H_X}VV   @VV{H_{g^{-1}(0)}^{\rm mon}}V
\\
\KK_0({\rm MHM}_X) @>>{\Psi_g}> 
\KK_0({\rm MHM}_{g^{-1}(0)}^{\rm mon}), 
\end{CD}
\end{equation}
where $\Psi_g$ is induced by the nearby cycles 
of mixed Hodge modules. 
\end{proposition}

\subsection{Equivariant Ehrhart theory of Katz and Stapledon}\label{polysec}

In this subsection, we recall the definition of 
some polynomials introduced by 
Katz-Stapledon~\cite{K-S-1} and Stapledon~\cite{Stapledon} 
in their Equivariant Ehrhart theory. 
Throughout this paper, we regard the empty set 
$\emptyset$ as a $(-1)$-dimensional polytope,
and as a face of any polytope.
Let $P$ be a polytope.
If a subset $F\subset P$ is a face of $P$, we write $F\prec P$.
For a pair of faces $F\prec F' \prec P$ of $P$,
we denote by $[F,F']$ the face poset $\{F''\prec 
P\mid F\prec F''\prec F'\}$,
and by $[F,F']^{*}$ a poset which is equal to 
$[F,F']$ as a set with the reversed order.

\begin{definition}
Let $B$ be the poset $[F,F']$ or $[F,F']^{*}$.
One defines a polynomial $g(B,t)$ of degree 
$\leq(\dim F' -\dim F)/2$ as follows.
If $F = F'$, one sets $g(B;t)=1$.
If $F \neq F'$ and $B=[F,F']$ (resp. $B=[F,F']^{*}$), 
the coefficients of $g(B;t)$ are uniquely 
determined by the equality 
\begin{align*}
t^{\dim{F'}-\dim{F}}g(B;t^{-1})=\sum_{F''\in[F,F']}
(t-1)^{\dim{F'}-\dim{F''}}g([F,F''];t).
\\ ({\rm resp.}~t^{\dim{F'}-\dim{F}}g(B;t^{-1})=
\sum_{F''\in[F,F']^{*}}(t-1)^{\dim{F''}-\dim{F}}g([F'',F']^{*};t).) 
\end{align*}
\end{definition}

In what follows, we assume that $P$ is a lattice polytope in $\RR^n$.
Let $S$ be a subset of $P\cap \ZZ^n$ containing the 
vertices of $P$, and $\omega \colon S\to \ZZ$ be a function.
We denote by $\UH_{\omega}$ the convex hull in $\RR^n\times \RR$ 
of the set $\{(v,s)\in \RR^n\times\RR \mid v\in S, s\geq \omega(v)\}$.
Then, the set of all the projections of the bounded faces of $\UH_{\omega}$ 
to $\RR^n$ defines a lattice polyhedral subdivision $\mathcal{S}$ of $P$.
Here a lattice polyhedral subdivision $\mathcal{S}$ 
of a polytope $P$ is a set of some polytopes in $P$
such that the intersection of any two polytopes in $\mathcal{S}$ is a 
face of both and all vertices of any polytope 
in $\mathcal{S}$ are in $\ZZ^{n}$. 
Moreover, the set of all the bounded faces of $\UH_{\omega}$ defines 
a piecewise $\QQ$-affine convex function $\nu\colon P\to \RR$.
For a cell $F\in\mathcal{S}$, we denote by 
$\sigma(F)$ the smallest face of $P$ containing $F$,
and $\LK_{\mathcal{S}}(F)$ the set of all cells 
of $\mathcal{S}$ containing $F$.
We call $\LK_{\mathcal{S}}(F)$ the link of $F$ in $\mathcal{S}$.
Note that $\sigma(\emptyset)=\emptyset$ and 
$\LK_{\mathcal{S}}(\emptyset)=\mathcal{S}$.

\begin{definition}
For a cell $F\in \mathcal{S}$, the $h$-polynomial 
$h(\LK_{\mathcal{S}}(F);t)$ 
of the link $\LK_{\mathcal{S}}(F)$ of $F$ is defined by
\[t^{\dim{P}-\dim{F}}h(\LK_{\mathcal{S}}(F);t^{-1})=
\sum_{F'\in \LK_{\mathcal{S}}(F)}g([F,F'];t)(t-1)^{\dim{P}-\dim{F'}}.\]
The local $h$-polynomial $l_{P}(\mathcal{S},F;t)$ 
of $F$ in $\mathcal{S}$ is defined by
\[l_{P}(\mathcal{S},F;t)=\sum_{\sigma(F)
\prec Q\prec P}(-1)^{\dim{P}-\dim Q}h(
\LK_{{\mathcal{S}}|_{Q}}(F);t) \cdot g([Q,P]^{*};t).\]
\end{definition}

For $\lambda\in \CC$ and 
$v\in mP\cap{\ZZ^n}$ ($m\in \ZZ_+:=\ZZ_{\geq 0}$) we set 
\begin{align*}
w_{\lambda}(v)=
\left\{
\begin{array}{ll}
1& \Bigl( \exp\ \bigl( 2\pi\sqrt{-1}\cdot m\nu(\frac{v}{m})\bigr)=
\lambda \Bigr) \\\\
0&( \text{otherwise} ). \\
\end{array}
\right.
\end{align*}
One defines the $\lambda$-weighted Ehrhart polynomial 
$f_{\lambda}(P,\nu;m)\in\ZZ[m]$ 
of $P$ with respect to $\nu\colon P\to \RR$ by
\[f_{\lambda}(P,\nu;m):=\sum_{v\in {mP}\cap\ZZ^{n}}w_{\lambda}(v).\]
Then $f_{\lambda}(P,\nu;m)$ is a polynomial in $m$ 
with coefficients $\ZZ$ 
whose degree is $\leq\dim P$ (see \cite{Stapledon}).  

\begin{definition}[{\cite{Stapledon}}]
\begin{enumerate}
\item One defines the $\lambda$-weighted $h^{*}$-polynomial 
$h^{*}_{\lambda}(P,\nu;u)
\in \ZZ[u]$ by 
\[\sum_{m\geq 0}f_{\lambda}(P,\nu;m)u^{m}=
\frac{h^{*}_{\lambda}(P,\nu;u)}{(1-u)^{\dim{P}+1}}.\]
If $P$ is the empty polytope, we set $h^{*}_{1}
(P,\nu;u)=1$ and $h^{*}_{\lambda}(P,\nu;u)=0~(\lambda\neq 1)$.
\item One defines the $\lambda$-local weighted 
$h^{*}$-polynomial $l^{*}_{\lambda}(P,\nu;u)\in\ZZ[u]$ by
\[l^{*}_{\lambda}(P,\nu;u)=\sum_{Q\prec P}
(-1)^{\dim{P}-\dim{Q}}h^{*}_{\lambda}(Q,\nu|_{Q};u)
 \cdot g([Q,P]^{*};u).\]
If $P$ is the empty polytope, we set $l^{*}_{1}(P,\nu;u)=1$ 
and $l^{*}_{\lambda}(P,\nu;u)=0~(\lambda\neq 1)$.
\end{enumerate}
\end{definition} 

\begin{definition}[{\cite{Stapledon}}]\label{def:poly}
\begin{enumerate}
\item One defines the $\lambda$-weighted limit mixed 
$h^{*}$-polynomial $h^{*}_{\lambda}(P,\nu;u,v)\in\ZZ[u,v]$ by
\[h^{*}_{\lambda}(P,\nu;u,v):=\sum_{F\in\MCS}v^{\dim{F}+1}
l^{*}_{\lambda}(F,\nu|_{F};uv^{-1}) \cdot h(\LK_{\MCS}(F);uv).\]
\item One defines the $\lambda$-local weighted limit mixed 
$h^{*}$-polynomial $l^{*}_{\lambda}(P,\nu;u,v)\in\ZZ[u,v]$ by
\[l^{*}_{\lambda}(P,\nu;u,v):=\sum_{F\in\MCS}v^{\dim{F}+1}
l^{*}_{\lambda}(F,\nu|_{F};uv^{-1}) \cdot l_{P}(\MCS,F;uv).\]
\item One defines the $\lambda$-weighted refined limit mixed 
$h^{*}$-polynomial $h^{*}_{\lambda}(P,\nu;u,v,w)\in\ZZ[u,v,w]$ by
\[h^{*}_{\lambda}(P,\nu;u,v,w):=\sum_{Q\prec P}
w^{\dim{Q}+1}l^{*}_{\lambda}(Q,\nu|_{Q};u,v) \cdot g([Q,P];uvw^2).\]
\end{enumerate}
\end{definition}

\section{Monodromies and limit mixed Hodge structures of families of 
hypersurfaces  in algebraic tori $(\CS)^n$}\label{sec:3}

Let $\BK=\CC(t)$ be the field of rational functions of $t$ and
$f(t,x)=\sum_{v \in \ZZ^n}a_{v}(t)x^{v} \in 
\BK[x_{1}^{\pm},\dots,x_{n}^{\pm}]\ (a_{v}(t)\in\BK)$
a Laurent polynomial of $x=(x_{1},\dots,x_{n})$ with coefficients in $\BK$.
For $v\in\ZZ^n$ by the Laurent expansion $a_{v}(t)=\sum_{j\in\ZZ}a_{v,j}
t^j\ (a_{v,j}\in\CC)$ of the rational function $a_{v}(t)$
we set
\[ o(v):=\ord_{t}a_{v}(t)=\min\{j \ | \ a_{v,j}\neq 0\}.\]
If $a_{v}(t) \equiv 0$ we set $o(v)=+\infty$.
Then we define an (unbounded) polyhedron $\UH_{f}$ in $\RR^{n+1}$ by
\[\UH_{f} = \Conv{\Bigl[ \bigcup_{v\in\ZZ^n} 
\{(v,s) \in \RR^{n+1} \ | \ s \geq o(v)\}  
\Bigr]} \subset \RR^{n+1},\]
where $\Conv{(\ \cdot\ )}$ stands for the convex hull.
Throughout this paper we assume that the dimension of $\UH_{f}$ is $n+1$.
Then by the projection $p:\RR^{n+1}=\RR^n\times\RR^{1}\twoheadrightarrow 
\RR^n$ we obtain an $n$-dimensional polytope $P:=p(\UH_{f}) 
\subset \RR^n$. 
We call it the Newton polytope of $f$. 
Set $\RR_+:= \RR_{\geq 0} \subset \RR$. 
Let $\Sigma_{0}$ be the dual fan of 
$\UH_{f}$ in $\RR^n\times\RR^{1}_{+}\subset \RR^{n+1}.$
We call its subfan $\Xi_{0}$ in 
$\RR^n\simeq \RR^n\times \{0\}$ consisting of cones 
$\sigma \in \Sigma_{0}$ contained in $\RR^n\times \{0\}\subset \RR^{n+1}$
the recession fan of $\UH_{f}$.
Let $\nu_{f}:P\to \RR$ be the function 
defining the bottom part of the boundary 
$\partial \UH_{f}$ of $\UH_{f}$ and $\mathcal{S}$ 
the subdivision of $P$ by the lattice polytopes
$p(\widetilde{F})\subset \RR^n\ (\widetilde{F}\prec \UH_{f})$.
Then for each cell $F$ in $\mathcal{S}$ the 
restriction $\nu_{f}|_{F}$ of $\nu_{f}$ 
to $F\subset P$ is an affine $\QQ$-linear 
function taking integral values on the vertices of $F$. 
Let us identify the affine subspace 
$\RR^n \times \{ 1 \} \subset \RR^{n+1}$ of 
$\RR^{n+1}$ with $\RR^n$ by the projection. 
Then by using the cones $\sigma \in \Sigma_0$ 
such that $\dim \sigma <n+1$ we define a 
polyhedral hypersurface ${\rm Trop}(Y)$ in 
$\RR^n \times \{ 1 \} \simeq \RR^n$ by 
\[ {\rm Trop}(Y) = \bigcup_{\dim \sigma <n+1} 
\Bigl\{ \sigma \cap (\RR^n \times \{ 1 \}) \Bigr\} \subset 
\RR^n \times \{ 1 \} \simeq \RR^n. \]
We call it the tropical variety of $Y$ 
(see \cite{R-S-T}). 
It has a decomposition 
\[ {\rm Trop}(Y) = \bigsqcup_{\dim \sigma <n+1} 
\Bigl\{ 
\relint \sigma \cap (\RR^n \times \{ 1 \}) 
\Bigr\} \]
into the (locally closed) 
cells $\relint \sigma \cap (\RR^n \times \{ 1 \})
\subset {\rm Trop}(Y)$. It is clear that 
there exists a one to one correspondence 
between the cells $F$ in $\mathcal{S}$ 
such that $\dim F >0$ and those in ${\rm Trop}(Y)$. 
In \cite{Stapledon} the author used the cell 
decomposition of 
${\rm Trop}(Y)$ to express the motivic nearby 
fiber $\psi_t([Y])$. However in this paper, 
we use only the subdivision $\mathcal{S}$ of $P$. 
For a cell $F$ in $\mathcal{S}$ by taking the (unique) 
compact face $\widetilde{F}\prec \UH_{f}$ 
of $\UH_{f}$ such that 
$F=p(\widetilde{F})$ we define the initial Laurent polynomial 
$I^{F}_{f}(x)\in\CC[x_{1}^{\pm},\dots, x_{n}^{\pm}]$ by
\[I^{F}_{f}(x) = \sum_{(v,j)\in\widetilde{F}}
a_{v,j}x^{v}\in\CC[x_{1}^{\pm},\dots, x_{n}^{\pm}].\]
By identifying $\Aff (F)\cap\ZZ^n$ with $\ZZ^{\dim{F}}$ 
we may consider $I^{F}_{f}(x)$
as a Laurent polynomial on the algebraic torus 
$T_{F}=\Spec(\CC[ \Aff (F) \cap\ZZ^n])\simeq (\CS)^{\dim F}$.
We denote by $V_{F}\subset T_{F}$ the 
hypersurface defined by $I^{F}_{f}(x)$ in $T_{F}$.
By the affine $\QQ$-linear extension $\nu_{F}:
\Aff(F)\to \RR$ of $\nu_{f}|_{F}$ to 
$\Aff (F) \simeq \RR^{\dim F}$ we define an 
element $e_{F}$ of the algebraic torus
$T_{F}=\Spec(\CC[ \Aff (F) \cap\ZZ^n])\simeq 
\Hom_{{\rm group}}( \Aff (F) \cap\ZZ^n,\CS)$ by
\[e_{F}(v)=\exp \Bigl(- 2\pi \sqrt{-1}\nu_{F}(v) \Bigr)\ 
\qquad (v\in \Aff (F) \cap \ZZ^n).\]
Then by the multiplication $\ell_{e_{F}}:
T_{F}\xrightarrow{\sim}T_{F}$ by $e_{F}\in T_{F}$
we have $\ell_{e_{F}}(V_{F})=V_{F}$.
We thus obtain a good action of $\MH$ on 
$V_F$ and an element $[V_{F}\circlearrowleft \MH]\in 
\M_{\CC}^{\hat{\mu}}$. 
The hypersurface $f^{-1}(0)\subset T=\CS_{t}\times
(\CS)^{n}_{x}$ defines a family $Y$ of 
hypersurfaces of $T_{0}=(\CS)^{n}_{x}$ over a small punctured disk
$B(0; \e )^*= \{t\in \CC \ | \ 0< |t| < \varepsilon \}\ (0<\varepsilon \ll1)$.
By the projection $\pi:T=\CS_{t}\times (\CS)^{n}_{x} 
\twoheadrightarrow \CS_{t}$, 
for $t\in \CC$ such that $0<|t|<\varepsilon$
we set $Y_{t}:= \pi^{-1}(t)\cap Y \subset \{t\} 
\times T_{0}\simeq T_{0} \simeq (\CS)^{n}_{x}$.

\begin{definition}\label{schon}
One says that the family $Y$ of hypersurfaces 
$\{Y_{t}\}_{0<|t| < \varepsilon}$ of 
$T_{0}\simeq (\CS)^{n}_{x}$ is sch\"{o}n
if for any cell $F$ in $\mathcal{S}$ the hypersurface 
$V_{F}\subset T_{F}$ of $T_{F}$ is smooth and reduced.
\end{definition}

For the family $Y$ over the punctured disk, denote by $\psi_{t}([Y])\in 
\M_{\CC}^{\hat{\mu}}$ its motivic nearby fiber by the function
$t= {\rm id}_{\CC}:\CC\to\CC$ (see Section \ref{sec:2}). 
Now let us prove Theorem \ref{RCS}.
\medskip

\begin{proof}[{\bf Proof of Theorem~\ref{RCS}}]
\ \\ \ Let $\Sigma$ be a smooth subdivision of the dual fan 
$\Sigma_{0}$ of ${\rm UH}_f$ and 
$\Xi$ its subfan in $\RR^n \simeq \RR^{n}
\times \{0\}$ consisting of cones $\sigma \in \Sigma$
contained in $\RR^{n}\times \{0\}\subset \RR^{n+1}$. 
We denote by $\Lambda$ the fan $\{\{0\} , \RR^{1}_{+} \}$ in $\RR^1$ 
formed by the faces of the closed half line 
$\RR^{1}_{+}$ in $\RR^1$.
Let $X_{\Sigma}$ (resp. $X_{\Xi}$) be the 
toric variety associated to $\Sigma$ (resp. $\Xi$). 
Recall that the algebraic torus $T=( \CC^*)^{n+1}$ acts on 
$X_{\Sigma}$. For a cone $\sigma \in \Sigma$ denote by 
$T_{\sigma} \simeq ( \CC^*)^{n+1- \dim{\sigma}}$ 
the $T$-orbit in $X_{\Sigma}$ associated to it. 
Then by the morphism $\Sigma \to \Lambda$ of fans induced by the 
projection $\RR^{n}\times \RR^{1}_{+} 
\twoheadrightarrow \RR^{1}_{+}$
we obtain a morphism
\[\pi_{\Sigma}:X_{\Sigma} \longrightarrow \CC\]
of toric varieties.
Restricting it to $\CS\subset\CC$, we obtain the projection 
$\CS\times X_{\Xi}\twoheadrightarrow\CS$ 
which extends naturally the previous one
$\pi:T=\CS\times T_{0}\twoheadrightarrow \CS$.
Let $\rho_{1},\dots,\rho_{N}\in \Sigma$ be the rays i.e. 
the $1$-dimensional 
cones in $\Sigma$. We may assume that for some 
$r \leq N$ we have $\rho_{i}\cap(\RR^n\times\{0\})=\{0\}$
$\iff$ $1 \leq i \leq r$.
For $1\leq i \leq r$ 
the $T$-orbit $T_i=T_{\rho_i} \subset X_{\Sigma}$ 
associated to $\rho_{i}$ 
satisfies the condition 
$\pi_{\Sigma}(T_{i})=\{0\}\subset\CC$.
We can easily see that for their 
closures $\overline{T_{i}}\subset X_{\Sigma}$
in $X_{\Sigma}$ we have
\[\pi_{\Sigma}^{-1}(\{0\}) =
 \bigcup_{i=1}^{r}\ov{T_{i}}.\]
For $1\leq i \leq N$ let $\alpha_i 
\in \rho_{i}\cap(\ZZ^{n+1}\setminus\{0\})$ 
be the primitive vector on the ray $\rho_{i}$.
We define a non-negative integer 
$b_{i}\geq0$ by $b_{i}=q( \alpha_i )$,
where $q:\RR^{n+1}=\RR^n\times \RR^{1} 
\twoheadrightarrow \RR^{1}$ is the projection.
Then it is easy to see that for  $1\leq i \leq r$ 
the order of the zeros of the function 
$t\circ \pi_{\Sigma}:X_{\Sigma}\to \CC$ along the
toric divisor $\ov{T_{i}}\subset X_{\Sigma}$ 
is equal to $b_{i}>0$. 
For a cone $\sigma \in \Sigma$ 
such that $\relint \sigma \subset 
\Int ( \RR^{n} \times \RR_+^1)$ 
we denote by $\tl{F_{\sigma}} 
\prec \UH_{f}$ its supporting face in $\UH_{f}$. 
Then by the sch\"{o}nness of $Y$, 
the closure $\ov{Y}$ of $Y\subset 
T \subset X_{\Sigma}$ in $X_{\Sigma}$ intersects the $T$-orbit
$T_{\sigma} \simeq (\CS)^{n+1-\dim{\sigma}}
\subset X_{\Sigma}$ transversally. 
This in particular implies that $\ov{Y}$ is 
smooth on a neighborhood of 
$\pi^{-1}_{\Sigma}(\{0\})$ in $X_{\Sigma}$. 
The resulting smooth hypersurface 
$W_{\sigma}:=\ov{Y}\cap 
T_{\sigma}\subset T_{\sigma}$ of $T_{\sigma}$ is
defined by the $\tl{F_{\sigma}}$-part 
$f_{\tl{F_{\sigma}}}(t,x)$ 
of $f(t,x)$. We can also easily see that for any cone 
$\sigma \in \Xi$ the hypersurface $\ov{Y}$ intersects 
$T_{\sigma}$ transversally on a neighborhood of 
$\pi_{\Sigma}^{-1}(\{0\})$ in $X_{\Sigma}$. 
Moreover the divisor $D= \ov{Y} \setminus Y= 
\ov{Y} \cap (X_{\Sigma} \setminus T)$ in $\ov{Y}$ 
is normal crossing. 
Let $\pi_{\ov{Y}}=
\pi_{\Sigma}|_{\ov{Y}}:\ov{Y}\to \CC$ be 
the restriction of $\pi_{\Sigma}$ to $\ov{Y}$. 
Then by Theorem \ref{GLMT} and $t \circ \pi_{\ov{Y}} 
= \pi_{\ov{Y}}$ we have 
\[
\psi_{t}([Y])= 
( \pi_{\ov{Y}} |_{\pi_{\ov{Y}}^{-1}(0)} )_!
( \mathcal{S}_{ \pi_{\ov{Y}} ,Y} ). 
\]
Moreover the morphism $( \pi_{\ov{Y}} |_{\pi_{\ov{Y}}^{-1}(0)} )_!$ 
sends $\mathcal{S}_{ \pi_{\ov{Y}} ,Y}$ to its 
underlying variety over the point 
$\{ 0 \} \simeq \Spec ( \CC )$, which we still 
denote by $\mathcal{S}_{ \pi_{\ov{Y}} ,Y}$ 
for short. Define an open subset $\Omega$ of 
$\ov{Y}$ containing $Y= \ov{Y} \cap T$ by 
\[
\Omega = 
\ov{Y} \cap \bigl( X_{\Sigma} \setminus 
\bigcup_{\sigma \in \Xi \setminus \{ 0\}} 
 \ \ov{T_{\sigma}} \bigr) 
\subset \ov{Y}. 
\]
Let $\Sigma'\subset \Sigma$ be the subset of $\Sigma$ 
consisting of cones $\sigma$ satisfying the condition 
$\sigma\cap(\RR^{n}\times \{0\})=\{0\}$. 
Then we have 
\[
\Omega = 
\ov{Y} \cap \bigl( \bigsqcup_{\sigma \in \Sigma'} 
 \ T_{\sigma} \bigr)
\]
and $\mathcal{S}_{ \pi_{\ov{Y}} ,Y}$ is described by 
some varieties over the normal crossing divisor 
\[
\Omega \cap \pi_{\ov{Y}}^{-1}(0) 
= \ov{Y} \cap \bigl( 
\bigsqcup_{\sigma \in \Sigma' \setminus \{ 0 \}} 
 \ T_{\sigma} \bigr) = \Omega \setminus Y.
\]
For a cone $\sigma \in \Sigma' \setminus \{ 0 \}$ 
let $\tl{W_{\sigma}}$ be the unramified Galois covering 
of $W_{\sigma}=\ov{Y} \cap T_{\sigma}$. 
From now, we shall prove that 
there exists an isomorphism 
\begin{equation}\label{maineq}
[\tl{W_{\sigma}}]= 
[V_{F_{\sigma}}\circlearrowleft \hat{\mu}]\cdot
 (\mathbb{L}-1)^{n+1-\dim{\sigma}-\dim{F_{\sigma}}},
\end{equation}
where we set $F_{\sigma}=p(\tl{F_{\sigma}})\in \mathcal{S}$. 
Assume first that $\dim\sigma+\dim F_{\sigma}=n+1$.
Let $\tau\in\Sigma$ be an $(n+1)$-dimensional cone such 
that $\sigma\prec \tau$ and set $l=\dim\sigma$.
After reordering the rays $\rho_{i}\ (1\leq i \leq N)$ we may assume that
\[\sigma=\RR_{+}\alpha_{1}+\dots+\RR_{+}\alpha_{l},\qquad \tau
=\RR_{+}\alpha_{1}+\dots+\RR_{+}\alpha_{n+1}.\]
Let us set $\sigma'=\RR_{+}\alpha_{l+1}+\dots+\RR_{+}
\alpha_{n+1}\prec\tau$ and
\[I=\{1,2,\dots,l\}, \qquad I'=\{l+1,l+2,\dots,n+1\}.\]
Moreover set $d_{I}=\gcd\{b_{i}\ |\ 1\leq i\leq l\}\geq 1$ and 
let $d_{I'}\geq 0$ be the generator of the subgroup 
$\ZZ b_{l+1}+\ZZ b_{l+2}+\dots+\ZZ b_{n+1}\subset \ZZ$.
We may define $\gcd\{b_{i}\mid l+1\leq i\leq n+1\}$ to be $d_{I'}$.
Then by the smoothness of the cone $\tau$ similarly we have 
$\gcd(d_{I},d_{I'})=\gcd\{b_{i}\mid 1\leq i\leq n+1\}=1$.
We shall prove the equality (\ref{maineq}) only in the case $d_{I'}\geq 1$.
In the case $d_{I'}=0\ (\iff \sigma'\subset \RR^{n}\times \{0\})$ 
we have $d_{I}=1$
and the proof is much easier.
Assume that $d_{I'}\geq 1$.
Let $\alpha^{*}_{1},\dots,\alpha^{*}_{n+1}\in(\ZZ^{n+1})^{*}\simeq \ZZ^{n+1}$ 
be the dual basis of $\alpha_{1},\dots,\alpha_{n+1}\in\ZZ^{n+1}$ and 
$\tau^{\vee}\subset (\RR^{n+1})^{*}\simeq \RR^{n+1}$ the dual 
cone of $\tau \subset \RR^{n+1}$.
Then the affine open subset $\CC^{n+1}(\tau)(\simeq \CC^{n+1})$ 
of $X_{\Sigma}$
is defined by
\begin{align*}
&\CC^{n+1}(\tau)=\Spec(\CC[\tau^{\vee}\cap\ZZ^{n+1}])
\simeq \Spec(\CC[\ZZ_{+}\alpha^{*}_{1}+\dots+\ZZ_{+}\alpha^{*}_{n+1}])\\
\simeq\ &\Spec(\CC[\xi_{1},\dots,\xi_{n+1}])\simeq \CC^{n+1}_{\xi}
\quad(\alpha^{*}_{i}\longleftrightarrow \xi_{i}).
\end{align*}
By the coordinates $\xi=(\xi_{1},\dots,\xi_{n+1})$ of $\CC^{n+1}(\tau)
\subset X_{\Sigma}$ its subset $T_{\sigma}=\Spec(\CC[\Aff(\sigma)^{\perp}
\cap\ZZ^{n+1}])\simeq (\CC^{*})^{n+1-l}$ is explicitly given by
\[T_{\sigma}=\{\xi\in\CC^{n+1}(\tau)\mid \xi_{i}=0\ (1\leq i\leq l),\  
\xi_{i}\neq 0\ (l+1\leq i \leq n+1)\}.\]
Moreover the restriction of the function 
$t\circ\pi_{\Sigma}\colon X_{\Sigma}
\rightarrow \CC$ to $\CC^{n+1}(\tau)$ is equal to 
$\xi^{b_{1}}_{1}\dots\xi_{n+1}^{b_{n+1}}
\in\CC[\xi_{1},\dots,\xi_{n+1}]$ which corresponds to the element 
$b_{1}\alpha^{*}_{1}+\dots+b_{n+1}\alpha^{*}_{n+1}$ in the group ring 
$\CC[\tau^{\vee}\cap\ZZ^{n+1}]$.
From this we see that the unramified Galois covering $\tl{W_{\sigma}}$ of 
$W_{\sigma}=\overline{Y}\cap T_{\sigma}$ is given by
\[\tl{W_{\sigma}}=\{(t,(\xi_{l+1},\dots,\xi_{n+1}))\in
\CC^{*}\times T_{\sigma}\mid t^{d_{I}}
=\xi^{-b_{l+1}}_{l+1}\dots \xi^{-b_{n+1}}_{n+1},\ 
f_{\tl{F_{\sigma}}}(\xi_{l+1},\dots,\xi_{n+1})=0\}.\]
Recall that the action of the group $\hat{\mu}$ on $[\tl{W_{\sigma}}]\in
\M_{\CC}^{\hat{\mu}}$ is defined by the multiplication of $\zeta_{d_{I}}=
\exp(2\pi\sqrt{-1}/d_{I})\in\CC$ to the coordinate $t$.
In order to rewrite $[\tl{W_{\sigma}}]$ we shall introduce a 
new basis of the lattice $\ZZ^{n+1}$.
First by our assumption on $\sigma$, the affine span $\Aff(\sigma)\simeq 
\RR^{l}$ of $\sigma$ intersects the hyperplane $\RR^{n}\times \{0\}\subset 
\RR^{n+1}$ transversally in $\RR^{n+1}$ and there exists a basis 
$\beta_{1},\dots,
\beta_{l}$ of the lattice $\Aff(\sigma)\cap \ZZ^{n+1}\simeq \ZZ^{l}$ such that 
$\beta_{1},\dots,\beta_{l-1}\subset \ZZ^n\times \{0\}\ (\iff q(\beta_{1})=\dots
=q(\beta_{l-1})=0)$ and $q(\beta_{l})=d_{I}$.
By the condition $\sigma'\not\subset \RR^n\times \{0\}$ there exists also 
a basis $\beta_{l+1},\dots,\beta_{n+1}$ of the lattice $\Aff(\sigma')\cap
\ZZ^{n+1}\simeq \ZZ^{n+1-l}$ such that $\beta_{l+1},\dots,\beta_{n}\in
\ZZ^n\times \{0\}$ and $q(\beta_{n+1})=d_{I'}$.
By $\RR^{n+1}=\Aff(\sigma)\oplus\Aff(\sigma')$ we thus obtain a basis 
$\beta_{1},\dots,\beta_{l},\beta_{l+1},\dots,\beta_{n+1}$ of the lattice 
$\ZZ^{n+1}=\{\Aff(\sigma)\cap\ZZ^{n+1}\}\oplus\{\Aff(\sigma')\cap\ZZ^{n+1}\}$.
For the dual basis $\alpha^{*}_{1},\dots,\alpha^{*}_{n+1}$ 
we have the decomposition
\[w_{1}+w_{2}=b_{1}\alpha^{*}_{1}+\dots+b_{n+1}\alpha^{*}_{n+1}=
\left(\begin{array}{c}
0\\
\vdots\\
0\\
1
\end{array}\right),\] 
where we set $w_{1}=b_{1}\alpha^{*}_{1}+\dots+b_{l}\alpha^{*}_{l}\subset 
\Aff(\sigma')^{\perp}\cap\ZZ^{n+1}\simeq \ZZ^{l}$ and $w_{2}=b_{l+1}
\alpha^{*}_{l+1}+\dots+b_{n+1}\alpha^{*}_{n+1}\in\Aff(\sigma)^{\perp}
\cap\ZZ^{n+1}\simeq \ZZ^{n+1-l}$.
Moreover by the construction of $\beta_{1},\dots,\beta_{n+1}$, 
for the dual basis $\beta^{*}_{1},\dots,\beta^{*}_{n+1}$ of 
it we have also
\begin{align}\label{3pagesiki}
d_{I}\beta^{*}_{l}+d_{I'}\beta^{*}_{n+1}=
\left(\begin{array}{c}
0\\
\vdots\\
0\\
1
\end{array}\right).
\end{align}
We thus obtain $w_{1}=d_{I}\beta^{*}_{l}, w_{2}=d_{I'}\beta^{*}_{n+1}$.
By the condition $\sigma'\not\subset \RR^n\times \{0\}$ we have $w_{2}\neq 0$.
It follows also from our assumption on $\sigma$ that the restriction 
of the projection $p\colon \RR^{n+1}\twoheadrightarrow\RR^n$ to 
$\Aff(\sigma)^{\perp}\simeq\RR^{n+1-l}$ is injective.
Hence we get $p(w_{2})\neq 0$ and the non-vanishing 
$p(w_{1})\neq 0$ follows from $p(w_{1})+p(w_{2})=0$.
Since the two vectors $p(w_{1})$ and $p(w_{2})$ in $\ZZ^n\subset \RR^n$ 
are divisible by $d_{I}$ and $d_{I'}$ respectively 
and $\gcd(d_{I},d_{I'})=1$, 
they are divisible also by $d_{I}d_{I'}$.
Let us show that the lattice vector
\[\frac{1}{d_{I}d_{I'}}p(w_{2})=\frac{1}{d_{I}}
p(\beta^{*}_{n+1})\in\ZZ^n\]
thus obtained is primitive.
Suppose that it is divisible again by an integer $d\geq 2$.
Then we have
\begin{align*}
0=\langle\beta_{l},w_{2}\rangle=\langle p(\beta_{l}),
p(w_{2})\rangle+q(\beta_{l})\cdot q(w_{2})
=\langle p(\beta_{l}),p(w_{2})\rangle+d_{I}d_{I'}q(\beta^{*}_{n+1}).
\end{align*}
This implies that
\[q(\beta^{*}_{n+1})=-\left\langle p(\beta_{l}) ,\ 
\frac{1}{d_{I}d_{I'}}p(w_{2})\right\rangle\]
is divisible by $d$.
So $\beta^{*}_{n+1}$ is also divisible by $d\geq 2$,
which contradicts the fact that $\beta^{*}_{n+1}$ is primitive.
It follows also from
\[{1=\langle\beta_{i},\beta^{*}_{i}\rangle=\langle p(\beta_{i}),
p(\beta^{*}_{i})\rangle\quad (l+1\leq i\leq n)}\]
that the vectors $p(\beta^{*}_{i})\in\ZZ^n\subset \RR^{n}\ 
(l+1\leq i\leq n)$
are primitive.
Note that $\beta^{*}_{l+1},\dots,\beta^{*}_{n+1}$ form a 
basis of the lattice
$\Aff(\sigma)^{\perp}\cap\ZZ^{n+1}\simeq \ZZ^{n+1-l}$.
For their projections by $p\colon \RR^{n+1}\twoheadrightarrow \RR^n$
we have the following result.
\begin{proposition}\label{page4prop}
The vectors $p(\beta^{*}_{l+1}),\dots,p(\beta^{*}_{n}),
\frac{1}{d_{I}}p(\beta^{*}_{n+1})
\in p(\Aff(\sigma)^{\perp})\cap\ZZ^n\simeq p(\Aff(
\tl{F_{\sigma}}))\cap\ZZ^n=
\Aff(F_{\sigma})\cap\ZZ^n$ form a basis of the lattice 
$p(\Aff(\sigma)^{\perp})\cap\ZZ^n\simeq \ZZ^{n+1-l}$.
\end{proposition}
\begin{proof}
We have seen that $p(\beta^{*}_{l+1}),\dots,p(\beta^{*}_{n}),
\frac{1}{d_{I}}p
(\beta^{*}_{n+1})$ are primitive.
First let us show that they are linearly independent over $\RR$.
Suppose that we have
\[\lambda_{l+1}p(\beta^{*}_{l+1})+\dots+\lambda_{n}p(\beta^{*}_{n})+
\lambda_{n+1}\frac{p(\beta^{*}_{n+1})}{d_{I}}=0\]
for some $\lambda_{i}\in\RR\ (l+1\leq i\leq n+1)$.
Then by taking the pairings with $\beta_{i}\ (l+1\leq i \leq n)$ we 
obtain $\lambda_{i}=0\ (l+1\leq i\leq n)$ and hence $\lambda_{n+1}=0$.
Next we show that they generate the lattice 
$p(\Aff(\sigma)^{\perp})\cap\ZZ^n$ over $\ZZ$.
For this we use the following result.

\begin{lemma}\label{page5lemma}
The vectors $p(\beta^{*}_{l+1}),\dots,p(\beta^{*}_{n})$ form a 
basis of the lattice
$\{\RR p(\beta^{*}_{l+1})\oplus\dots\oplus\RR p(\beta^{*}_{n})\}
\cap\ZZ^n.$
\end{lemma}
\begin{proof}
Suppose that they do not generate the lattice over $\ZZ$.
Then there exist integers $d\geq 2$ and $\lambda_{l+1},\dots,
\lambda_{n}\in\ZZ$
such that
\[\frac{1}{d}\left\{\lambda_{l+1}p(\beta^{*}_{l+1})+\dots+
\lambda_{n}p(\beta^{*}_{n})\right\}\in\ZZ^n\]
and $\frac{\lambda_{i_{0}}}{d}\in\QQ\setminus\ZZ$ for some 
$l+1\leq i_{0}\leq n$.
By taking the pairing with $p(\beta_{i_{0}})$ we obtain 
$\frac{\lambda_{i_{0}}}{d}\in\ZZ$,
which is a contradiction.
\end{proof}
Let us continue the proof of Proposition~\ref{page4prop}.
By Lemma~\ref{page5lemma}, if the vectors $p(\beta^{*}_{l+1}),
\dots,p(\beta^{*}_{n}),\frac{1}{d_{I}}p(\beta^{*}_{n+1})$ do not 
generate the lattice $p(\Aff(\sigma)^{\perp})\cap\ZZ^n$,
then there exist integers $d\geq 2$ and $\lambda_{l+1},\dots,
\lambda_{n}\in\ZZ$ such that
\[\frac{1}{d}\left\{\frac{1}{d_{I}}p(\beta^{*}_{n+1})-\lambda_{l+1}p(
\beta^{*}_{l+1})-\dots-\lambda_{n}p(\beta^{*}_{n})\right\}\in\ZZ^n.\]
Set $\gamma^{*}=\beta^{*}_{n+1}-\lambda_{l+1}d_{I}\beta^{*}_{l+1}-
\dots-\lambda_{n}d_{I}\beta^{*}_{n}$.
Then we obtain a new basis $\beta^{*}_{l+1},\dots,\beta^{*}_{n},\gamma^{*}$ 
of the lattice $\Aff(\sigma)^{\perp}\cap\ZZ^{n+1}\simeq \ZZ^{n+1-l}$.
By taking the pairing with $\beta_{l}\in\Aff(\sigma)\cap\ZZ^{n+1}$ we get
\[0=\langle\beta_{l},\gamma^{*}\rangle=\langle p(\beta_{l}),
p(\gamma^{*})\rangle+d_{I}\cdot q(\gamma^{*}).\]
Since the lattice vector $p(\gamma^{*})\in\ZZ^n$ is divisible by 
$dd_{I}$, the integer $q(\gamma^{*})\in\ZZ$ and hence 
$\gamma^{*}\in\ZZ^{n+1}$ itself is so.
This contradicts the fact that $\gamma^{*}$ is primitive.
\end{proof}

Now we return to the proof of Theorem~\ref{RCS}.
By the new basis $\beta^{*}_{l+1},\dots,\beta^{*}_{n+1}$ of the lattice
 $\Aff(\sigma)^{\perp}\cap\ZZ^{n+1}\simeq \ZZ^{n+1-l}$ 
we have an isomorphism
\begin{align*}
&T_{\sigma}=\Spec(\CC[\Aff(\sigma)^{\perp}\cap\ZZ^{n+1}])
\simeq \Spec(\CC[\ZZ\beta^{*}_{l+1}+\dots+\ZZ\beta^{*}_{n+1}])\\
\simeq \ &\Spec(\CC[z_{l+1},\dots,z_{n+1}])\quad (\beta^{*}_{i}
\longleftrightarrow z_{i}).
\end{align*}
By the new coordinates $z=(z_{l+1},\dots,z_{n+1})$ of 
$T_{\sigma}\simeq (\CC^{*})^{n+1-l}$ we have
\[\tl{W_{\sigma}}=\{(t,(z_{l+1},\dots,z_{n+1}))\in\CC^{*}\times 
T_{\sigma}\mid t^{d_{I}}=z^{-d_{I'}}_{n+1},\ 
f_{\tl{F_{\sigma}}}(z)=0\}.\]
On the other hand, by taking the projection $q\colon \RR^{n+1}
\twoheadrightarrow\RR^{1}$ of the both sides of the 
decomposition (\ref{3pagesiki}),
we obtain
\begin{align}\label{6pagesiki}
d_{I}\cdot q(\beta^{*}_{l})+d_{I'}\cdot q(\beta^{*}_{n+1})=1.
\end{align}
Note that for the basis $p(\beta^{*}_{l+1}),
\dots,p(\beta^{*}_{n}),\frac{1}{d_{I}}p
(\beta^{*}_{n+1})$ of the lattice 
$\Aff(F_{\sigma})\cap\ZZ^{n}$ constructed in 
Proposition~\ref{page4prop}
we have $\nu_{f}(p(\beta^{*}_{l+1})),\dots,\nu_{f}(p(\beta^{*}_{n}))\in\ZZ$ 
and $\nu_{f}(\frac{1}{d_{I}}p(\beta^{*}_{n+1}))\equiv
\frac{q(\beta^{*}_{n+1})}{d_{I}} \bmod \ZZ$.
By Proposition~\ref{page4prop}, we obtain an isomorphism
\[V_{F_{\sigma}}\simeq \{(s,(z_{l+1},\dots,z_{n+1}))\in\CC^{*}\times 
T_{\sigma}\mid s^{d_{I}}=z_{n+1},\ f_{\tl{F_{\sigma}}}(z)=0\}.\]
Moreover the action of the group $\hat{\mu}$ on ${[V_{F_{\sigma}}
\circlearrowleft \hat{\mu}]}\in\M_\CC^{\hat{\mu}}$ corresponds to 
the multiplication of $\exp(-2\pi\sqrt{-1}q(
\beta^{*}_{n+1})/d_{I})\in\CC$ to the coordinate $s$.
By the equality (\ref{6pagesiki}) we have
\begin{align}\label{7pagesiki}
\exp\left(-\frac{2\pi\sqrt{-1}
q(\beta^{*}_{n+1})}{d_{I}}\right)^{-d_{I'}}=
\exp\left(\frac{2\pi\sqrt{-1}}{d_{I}}\right)=\zeta_{d_{I}}.
\end{align}
Furthermore by $\gcd(d_{I},d_{I'})=1$ the morphism $V_{F_{\sigma}}
\rightarrow \tl{W_{\sigma}}$ defined by $(s,z)\longmapsto 
(s^{-d_{I'}},z)$ is an isomorphism.
It is compatible with the actions of $\hat{\mu}$ on the both 
sides by the equality (\ref{7pagesiki}).
We thus obtained the required isomorphism ${[\tl{W_{\sigma}}]}
\simeq{[V_{F_{\sigma}}\circlearrowleft \hat{\mu}]}$ in $\M_\CC^{\hat{\mu}}$.
If $\dim{\sigma}+\dim{F_{\sigma}}<n+1$, similarly we 
can prove an isomorphism
\[\tl{W_{\sigma}}\simeq V_{F_{\sigma}}\times (\CC^{*})^{n+1-
\dim{\sigma}-\dim{F_{\sigma}}},\]
but the action of $\hat{\mu}$ on the second factor 
$(\CC^{*})^{n+1-\dim\sigma-\dim{F_{\sigma}}}$ of the 
right hand side might be non-trivial.
Nevertheless by \cite[Example~2.2]{Stapledon} (witch 
follows essentially from the definition of 
$\M_\CC^{\hat{\mu}}$) we obtain an isomorphism
\[[\tl{W_{\sigma}}]\simeq [V_{F_{\sigma}}\circlearrowleft 
\hat{\mu}]\cdot (\mathbb{L}-1)^{n+1-\dim\sigma-\dim{F_{\sigma}}}\]
in $\M_\CC^{\hat{\mu}}$.
Then by Theorem \ref{GLMT} we have 
\begin{align*}
\psi_{t}([Y])& = \sum_{\sigma 
\in \Sigma' \setminus \{ 0\} }[\tl{W_{\sigma}}]
\cdot (1-\mathbb{L})^{\dim{\sigma}-1}
\\
& =\sum_{\sigma \in \Sigma' \setminus \{ 0\} }
(-1)^{n+1-\dim{\sigma}-
\dim{F_{\sigma}}}[V_{F_{\sigma}}
\circlearrowleft\hat{\mu}]\cdot
(1-\mathbb{L})^{n-\dim{F_{\sigma}}}.
\end{align*}
For a cell $F\in\mathcal{S}$ denote by $\tl{F}\prec \mathrm{UH}_{f}$ 
the unique compact 
face of $\mathrm{UH}_{f}$ such that $p(\tl{F})=F$ and let 
$F^{\circ} \in\Sigma_{0}$
be the cone which corresponds to it in the dual fan $\Sigma_{0}$.
Then we can easily show
\begin{align*}
\sum_{\substack{\sigma \in \Sigma' \setminus \{ 0\} 
\\\mathrm{rel.int}\sigma\subset 
\mathrm{rel.int} F^{\circ} }}(-1)^{n+1-\dim
\sigma-\dim{F_{\sigma}}}{[
V_{F_{\sigma}}\circlearrowleft \hat{\mu}]}
\cdot(1-\mathbb{L})^{n-\dim{F_{\sigma}}}\\
=\left\{
\begin{array}{l}
[V_{F}\circlearrowleft \hat{\mu}]\cdot (1-\mathbb{L})^{n-
\dim{F}}\qquad(\mathrm{rel.int}F\subset \Int P)\\\\
0 \qquad(\mathrm{otherwise})
\end{array}
\right.
\end{align*}
(cf. the proof of Matsui-Takeuchi 
\cite[Theorem 5.7 and Proposition 5.5]{M-T-4}) .
We thus obtain the desired formula 
\[ \psi_{t}([Y])=\sum_{\relint{F}\subset\Int{P}}[V_{F}
\circlearrowleft\hat{\mu}]\cdot(1-\mathbb{L})^{n-\dim{F}}. 
\]
This completes the proof. 
\end{proof}

\begin{remark}
As is clear from the above proof of 
Theorem \ref{RCS}, it can be immediately 
generalized to any sch\"{o}n family of 
subvarieties of $T_0 \simeq (\CS)^n$ 
as in \cite[Theorem 3.2 and Corollary 5.3]{Stapledon}. 
However in this paper, we do not use such a 
generalization. 
\end{remark}

By the proof of the above theorem we obtain the following result.
\begin{lemma}\label{smooth}
Assume that the family $Y$ is sch\"on.
Then there exists $\varepsilon>0$ such that the hypersurface 
$Y_{t}=Y\cap\pi^{-1}(t)\subset T_{0}=(\CS)^{n}_{x}$
 is Newton non-degenerate (see \cite{Oka}) 
for any $t\in\CS$ satisfying the condition $0<|t|<\varepsilon$.
\end{lemma}

Recall that for 
a constructible sheaf $\mathcal{F}\in \BDC_{c}(\CC)$ on $\CC$ 
its nearby cycle sheaf 
$\psi_{t}(\mathcal{F})\in \BDC_{c}(\{0\})$ by 
the function $t$ has 
a direct sum decomposition
\[\psi_{t}(\mathcal{F}) = \bigoplus_{\lambda\in\CC}
\psi_{t,\lambda}(\mathcal{F})\]
with respect to the generalized eigenspaces 
$\psi_{t,\lambda}(\mathcal{F})\in \BDC_{c}(\{0\})$ 
for eigenvalues $\lambda\in\CC$ (see Dimca~\cite{Dimca}). 
Let $j: B(0; \e )^* \hookrightarrow B(0; \e )= 
B(0; \e )^* \sqcup \{ 0 \}$ be the inclusion map. 
By the rotation on the punctured disk we obtain 
the monodromy automorphisms 
\[\Phi_{j}:H^{j}_{c}(Y_{t};\CC)\xrightarrow{\sim}
H^{j}_{c}(Y_{t};\CC) \qquad (j\in\ZZ)\]
for $0<|t|<\varepsilon$.
For $\lambda\in \CC$ let 
\[H^{j}_{c}(Y_{t};\CC)_{\lambda} \subset H^{j}_{c}(Y_{t};\CC)\] 
be the generalized eigenspace of $\Phi_{j}$ for the eigenvalue $\lambda$. 
Then we have an isomorphism
\[H^{j}\psi_{t,\lambda}(j_{!}R\pi_{!}\CC_{Y})\simeq 
H^{j}_{c}(Y_{t};\CC)_{\lambda}\]
for any $j \in \ZZ$ and $\lambda\in\CC$.

\begin{proposition}\label{Gysin}
Assume that the family $Y$ is sch\"on.
Then for $t\in\CS$ such that $0<|t|\ll1$ we have
\[H^{j}_{c}(Y_{t};\CC)\simeq0  \qquad (j<n-1)\]
and the Gysin map
\[H^{j}_{c}(Y_{t};\CC) \longrightarrow H^{j+2}_{c}(T_{0};\CC)\]
associated to the inclusion map $Y_{t}\hookrightarrow T_{0}$
is an isomorphism (resp. surjective) for $j>n-1$ (resp. $j=n-1$). 
Moreover the monodromy $\Phi_{j}: H^{j}_{c}(Y_{t};\CC)
\xrightarrow{\sim} H^{j}_{c}(
Y_{t};\CC)$ is identity for any $j>n-1$. In particular, 
for any $\lambda \not= 1$ and $t\in\CS$ such that $0<|t|\ll1$ 
we have the concentration 
\[H^{j}_{c}(Y_{t};\CC)_{\lambda} \simeq 0  \qquad (j \not= n-1).\]
\end{proposition}
\begin{proof}
Since $Y_{t}\subset T_{0}$ is smooth and affine, 
the first assertion follows from 
the (generalized) Poincar\'e duality theorem
\[H^{j}_{c}(Y_{t};\CC)\simeq \lbrack{H^{2n-2-j}
(Y_{t};\CC)}\rbrack^{*} \qquad (j\in\ZZ).\]
Moreover by the above lemma, the second assertion 
follows from the weak Lefschetz 
theorem (see Danilov-Khovanskii \cite[Proposition 3.9]{D-K}) 
for the Newton non-degenerate hypersurface $Y_{t}
\subset T_{0}=(\CS)^{n}_{x}$. 
By the definition of the Gysin maps, we thus obtain 
isomorphisms
\[ H^{j}(T_t ;\CC) \simeq 
H^{j}(T_{0};\CC) \simto  H^{j}(Y_{t};\CC) \]
for $j<n-1$. Since the monodromy 
automorphisms on $H^{j}(T_t ;\CC) \simeq
H^{j}(T_{0};\CC)$ are trivial, for $j<n-1$ 
that on $H^{j}(Y_{t};\CC)\simeq \lbrack{H^{2n-2-j}_{c} 
(Y_{t};\CC)}\rbrack^{*}$ is identity. 
\end{proof}

For a cell $F\in\mathcal{S}$ by using 
the affine linear extension $\nu_{F}\colon 
\mathrm{Aff}(F)\simeq \RR^{\dim{F}}\longrightarrow \RR$ 
of $\nu|_{F}\colon 
F\longrightarrow \RR$ we define
a positive integer $m_{F}$ to be the minimal 
one $m$ for which $m\cdot\nu_{F}$ takes only 
integer values on $\mathrm{Aff}(F)\cap\ZZ^{n}$.
It is clear that if $F\prec F^{\prime}$ the number $m_{F^{\prime}}$ is divisible by $m_{F}$.
Then we define a finite subset $R_{f}\subset \CC$ by
\[R_{f}=\bigcup_{F\subset\partial P}\{\lambda\in\CC\ |\ 
\lambda^{m_{F}}=1\}\subset \CC.\] 
Note that we have $1 \in R_f$. 

\begin{theorem}\label{th:2}
Assume that the family $Y$ is sch\"on.
Then for any $\lambda\notin R_{f}$ the morphism 
\[\psi_{t,\lambda}(j_{!}R\pi_{!}\CC_{Y})
\longrightarrow \psi_{t,\lambda}(j_{!}R\pi_{*}\CC_{Y})\]
induced by the one $R\pi_{!}\CC_{Y}\to R\pi_{*}\CC_{Y}$ 
is an isomorphism.
\end{theorem}

\begin{proof}
The proof is similar to that of \cite[Theorem 4.5]{T-T}. 
We shall use the notations $\Sigma$, $\Xi$, $X_{\Sigma}$, 
$\pi_{\Sigma}: X_{\Sigma} \rightarrow \CC$, 
$\ov{Y}$, $\pi_{\ov{Y}} :\ov{Y}\to \CC$ etc. 
in the proof of Theorem \ref{RCS}. 
By the sch\"onness of $Y$ the 
hypersurface $\ov{Y}\subset X_{\Sigma}$ 
intersects $T$-orbits in 
$\pi_{\Sigma}^{-1}(\{0\})$ transversally.
Recall that for any cone $\sigma\in \Xi$ the hypersurface 
$\ov{Y}$ intersects the $T$-orbit 
$T_{\sigma}$ associated to it transversally on a 
neighborhood of $\pi_{\Sigma}^{-1}(\{0\})$ in $X_{\Sigma}$. 
Recall also that the divisor $D= \ov{Y} \setminus Y= 
\ov{Y} \cap (X_{\Sigma} \setminus T)$ in $\ov{Y}$ 
is normal crossing there. 
Let $i_{D}:D \hookrightarrow \ov{Y}$ 
and $j_{Y}:Y \hookrightarrow \ov{Y}$ be the inclusion 
maps. Then there exists a commutative diagram
\[
\xymatrix{
Y \ar@{^{(}-{>}}[r]^{j_{Y}} \ar[d]_{\pi_Y= \pi|_Y}& 
\ov{Y} \ar[d]^{\pi_{\ov{Y}}} \\
\CS \ar@{^{(}-{>}}[r]^{j} & \CC.
}
\]
Since $\pi_{\ov{Y}} :\ov{Y}\to \CC$ is proper,
we have isomorphisms 
\[
\left\{
\begin{array}{l}
\psi_{t}(j_{!}R\pi_{!}\CC_{Y}) \simeq\psi_{t}
(R(\pi_{\ov{Y}})_{*}(j_{Y})_{!}\CC_{Y}),\\
\psi_{t}(j_{!}R\pi_{*}\CC_{Y})\simeq \psi_{t}
(Rj_{*}R ( \pi_Y)_{*}\CC_{Y}) \simeq\psi_{t}(
R(\pi_{\ov{Y}})_{*}R(j_{Y})_{*}\CC_{Y}).
\end{array}
\right.
\]
Therefore, by applying the functor $\psi_{t}\circ 
R(\pi_{\ov{Y}})_{*}:\BDC( \ov{Y} )
\to \BDC(\{0\})$ to the distinguished triangle
\[(j_{Y})_{!}\CC_{Y}\to R(j_{Y})_{*}\CC_{Y}\to 
(i_{D})_{*}(i_{D})^{-1}R(j_{Y})_{*}\CC_{Y} \xrightarrow{+1}\]
we obtain the new one
\[\psi_{t}(j_! R\pi_{!}\CC_{Y})\to \psi_{t}(j_! R\pi_{*}\CC_{Y})
\to \psi_{t}(R(\pi_{\ov{Y}})_{*} (i_{D})_{*} (i_{D})^{-1}
R(j_Y)_{*}\CC_{Y})\xrightarrow{+1}. \]
This implies that for the proof of the theorem it 
suffices to prove the vanishing
\[\psi_{t,\lambda}(
R(\pi_{\ov{Y}})_{*} (i_{D})_{*} (i_{D})^{-1}
R(j_Y)_{*}\CC_{Y}
 )\simeq 0\]
for any $\lambda \notin R_{f}$.
Since $\pi_{\ov{Y}} :\ov{Y}\to \CC$ is proper, 
by \cite[Proposition 4.2.11]{Dimca} and 
\cite[Exercise VIII.15]{K-S} 
for any $\lambda \in \CC$ we have an isomorphism 
\[\psi_{t,\lambda}(
R(\pi_{\ov{Y}})_{*} (i_{D})_{*} (i_{D})^{-1}
R(j_Y)_{*}\CC_{Y}
 )\simeq 
R \Gamma ( \pi_{\ov{Y}}^{-1}( \{ 0 \} ); 
\psi_{\pi_{\ov{Y}},\lambda}( (i_{D})_{*} (i_{D})^{-1}
R(j_Y)_{*}\CC_{Y})). 
\]
Now let us set 
\[ D^{\prime}=
\ov{D \setminus \pi_{\ov{Y}}^{-1}( \{ 0 \} )} 
= \bigcup_{\sigma \in \Xi} ( 
\ov{Y} \cap \ov{T_{\sigma}} ) 
\subset D
\]
and $Y^{\prime}= \ov{Y} \setminus D^{\prime} 
\supset Y$. Then $D^{\prime}$ is a normal 
crossing divisor of $\ov{Y}$ 
on a neighborhood of 
$\pi^{-1}_{\ov{Y}}(\{0\})$ and for the 
inclusion maps 
$i_{D^{\prime}}:D^{\prime} \hookrightarrow \ov{Y}$ 
and $j_{Y^{\prime}}:Y^{\prime} \hookrightarrow \ov{Y}$ 
there exists an isomorphism 
\[ 
\psi_{\pi_{\ov{Y}},\lambda}( (i_{D})_{*} (i_{D})^{-1}
R(j_Y)_{*}\CC_{Y}) \simeq 
\psi_{\pi_{\ov{Y}},\lambda}(
 (i_{D^{\prime}})_{*} (i_{D^{\prime}})^{-1}
R(j_{Y^{\prime}})_{*}\CC_{Y^{\prime}}). 
\]
Note that for the natural stratification of 
the normal crossing divisor $D^{\prime} \subset 
\ov{Y}$ by the strata 
\[
D^{\prime}_{\sigma} := 
\ov{Y} \cap \bigl( \ov{T_{\sigma}} \setminus 
\bigcup_{\tau \in \Xi, \ \sigma \precneqq \tau} 
 \ \ov{T_{\tau}} \bigr) 
\subset D^{\prime} \qquad 
( \sigma \in \Xi \setminus \{ 0 \}) 
\]
the cohomology sheaves 
$H^j (i_{D^{\prime}})^{-1}
R(j_{Y^{\prime}})_{*}\CC_{Y^{\prime}}$ 
($j \in \ZZ$) are constructible. 
Moreover their restrictions to 
each stratum $D^{\prime}_{\sigma} \subset D^{\prime}$ 
are constant. Then by cutting the support of 
the complex $(i_{D^{\prime}})^{-1}
R(j_{Y^{\prime}})_{*}\CC_{Y^{\prime}} 
\in \BDC_{c}( D^{\prime} )$ by the 
stratification and truncating each of 
the resulting complexes, it suffices to 
prove the vanishing 
\[
R \Gamma ( \pi_{\ov{Y}}^{-1}( \{ 0 \} ); 
\psi_{\pi_{\ov{Y}},\lambda}( 
\CC_{\ov{Y} \cap \ov{T_{\sigma}}} ) )
\simeq 0. 
\]
for any $\lambda \notin R_{f}$ and 
any cone $\sigma \in \Xi \setminus \{ 0 \}$. 
Fixing such $\lambda$ and $\sigma$ we 
shall prove the vanishing from now. Set 
$ \pi_{\sigma}= \pi_{\ov{Y}}|_{\ov{Y} 
\cap \ov{T_{\sigma}}} : 
\ov{Y} \cap \ov{T_{\sigma}} \longrightarrow \CC$ 
and $\F_{\sigma}= \CC_{\ov{Y} \cap \ov{T_{\sigma}}}$. 
Then it is enough to show the vanishing 
\[
R \Gamma ( \pi_{\sigma}^{-1}( \{ 0 \} ); 
\psi_{\pi_{\sigma},\lambda}( \F_{\sigma} ) ) 
\simeq 0. 
\]
By a classical result on 
the decompositions of nearby cycle 
perverse sheaves associated to normal crossing 
divisors (see e.g. 
Dimca-Saito \cite[Section 1.4]{D-S}), 
each graded piece of the nearby cycle 
perverse sheaf $\psi_{\pi_{\sigma},\lambda}( \F_{\sigma} ) 
[n- \dim \sigma -1] 
\in \BDC_{c}( \pi_{\sigma}^{-1}( \{ 0 \} ) )$
of a filtration 
has a decomposition into 
some mininal extension perverse sheaves of 
$\SL_{\tau} [n- \dim \tau ] \in \BDC_{c}(
 \ov{Y} \cap T_{\tau} )$, where 
$\tau$ is a cone in $\Sigma \setminus \Xi$ 
such that $\sigma \prec \tau$ and 
$\SL_{\tau}$ is a rank one local system on 
the ($n- \dim \tau$)-dimensional 
smooth hypersurface 
$\ov{Y} \cap T_{\tau} \subset T_{\tau} 
\simeq ( \CC^*)^{n- \dim \tau +1}$. 
Moreover the cones $\tau$ appearing in 
this decomposition should satisfy the 
following condition. Let $\tau$ be a cone 
in $\Sigma \setminus \Xi$ such that 
$\sigma \prec \tau$ and let 
$\alpha_1, \ldots, \alpha_k \in \tau \cap 
( \ZZ^{n+1} \setminus \{ 0 \} )$ be the 
primitive vectors on the edges of 
$\tau$ not contained in the hyperplane 
$\RR^n \times \{ 0 \} \subset \RR^{n+1}$. 
Set $b_i=q( \alpha_i) >0$ ($1 \leq i \leq k$) 
and $d_{\tau}= {\rm gcd} \{ b_i \ | \ 
1 \leq i \leq k \} >0$. Then, for any 
$1 \leq i \leq k$ the order of the 
zero of the function 
$\pi_{\sigma}$ along the 
divisor of $\ov{Y} \cap \ov{T_{\sigma}}$ 
associated to the ray containing 
$\alpha_i$ is equal to $b_i>0$. Moreover 
for the cone 
$\tau$ to appear in the decomposition 
it should satisfy the condition 
$\lambda^{d_{\tau}}=1$. In such a case, 
the rank one local system $\SL_{\tau}$ 
on $\ov{Y} \cap T_{\tau}$ has the 
following condition. Let $\rho \in \Sigma 
\setminus \Xi$ be a ray such that 
$\tau \cap \rho = \{ 0 \}$ and 
$\tau ( \rho ) := \tau + \rho$ is a cone 
in $\Sigma$. Let $\beta \in \rho \cap 
( \ZZ^{n+1} \setminus \{ 0 \} )$ be the 
primitive vector on it. Then for any 
such $\rho$ the monodromy of the local 
system $\SL_{\tau}$ around the divisor 
$\ov{Y} \cap \ov{T_{\tau ( \rho ) }} \subset 
\ov{Y} \cap \ov{T_{\tau}}$ is given by 
the multiplication by the complex 
number $\lambda^{-q( \beta )} \in \CC$. 
By cutting the supports of 
the minimal extension perverse sheaves 
by the toric 
stratifications of $\ov{T_{\tau}}$, 
for the proof of the theorem it 
suffices to prove the vanishing 
\[ 
R \Gamma_c ( \ov{Y} \cap T_{\tau}; 
\SL_{\tau} ) \simeq 0
\]
for any cone $\tau$ in $\Sigma \setminus \Xi$ 
such that $\sigma \prec \tau$ and 
$\lambda^{d_{\tau}}=1$. 
For a cell $F$ in $\mathcal{S}$ let 
$\widetilde{F}\prec \UH_{f}$ be the  
unique compact face 
of $\UH_{f}$ such that $F=p(\widetilde{F})$ 
and $F^{\circ} \in \Sigma_0$ the 
cone which corresponds to it 
in the dual fan $\Sigma_0$. Then for 
the cone $\tau$ there exists a unique 
cell $F \in \mathcal{S}$ such that 
$F \subset \partial P$ and 
$\relint \tau \subset \relint F^{\circ}$. 
Let $\tau^{\prime}$ be a cone in 
$\Sigma \setminus \Xi$ 
such that $\tau^{\prime} \subset F^{\circ}$, 
$\dim \tau^{\prime} = \dim F^{\circ}$ and 
$\tau \prec \tau^{\prime}$. 
Then by the smoothness of the cone 
$\tau^{\prime}$ we can easily show the 
equality $m_F=d_{\tau^{\prime}}$. By our 
assumption $\lambda \notin R_f$ we thus 
obtain $\lambda^{d_{\tau^{\prime}}} \not= 1$. 
This implies that there exists a ray $\rho$ of 
$\tau^{\prime}$ such that 
$\tau \cap \rho = \{ 0 \}$ and 
the primitive vector $\beta \in \rho \cap 
( \ZZ^{n+1} \setminus \{ 0 \} )$ on it 
satisfies the condition 
$\lambda^{-q( \beta )} \not= 1$. 
Moreover, since the supporting faces 
of $\tau$ and $\tau^{\prime}$ in $\UH_{f}$ 
coincide and are equal to $\widetilde{F}
\prec \UH_{f}$, we have a product decomposition 
\[
\ov{Y} \cap T_{\tau} \simeq Z \times ( \CC^*)^k 
\]
for some variety $Z$ and $k>0$ such that 
the equation of the divisor 
$\ov{T_{\tau ( \rho )}} \subset \ov{T_{\tau}}$ 
corresponds to a coordinate of the torus 
$( \CC^*)^k $. 
Now the desired vanishing follows from 
the K\"unneth formula. 
This completes the proof.
\end{proof}

\begin{corollary}\label{CoEQEEP} 
Assume that the family $Y$ is sch\"on.
Then for any $\lambda\notin R_{f}$ and $t\in \CS$ such that 
$0<|t|\ll1$ we have the concentration 
\[H^{j}_{c}(Y_{t};\CC)_{\lambda}\simeq 0 \qquad (j\neq n-1).\]
Moreover for such $\lambda$ and $t$ 
the filtration on $H^{j}_{c}(Y_{t};\CC)_{\lambda}$ 
induced by the weight filtration of 
Deligne's mixed Hodge structure on 
$H^{j}_{c}(Y_{t};\CC)$ is concentrated in degree $n-1$.
\end{corollary}

\begin{proof}
Assume that $\lambda\notin R_{f}$.
Then the first assertion is already shown in 
Proposition \ref{Gysin}. However we shall give 
a new proof to it by using Theorem~\ref{th:2}. 
By Theorem~\ref{th:2} for $t\in\CC$ such that 
$0<|t|\ll1$ there exist isomorphisms
\[H^{j}_{c}(Y_{t};
\CC)_{\lambda}\xrightarrow{\sim}H^{j}(Y_{t};\CC)_{\lambda}
\qquad (j\in\ZZ).\]
Since the $(n-1)$-dimensional variety 
$Y_t$ is affine and smooth, 
by the (generalized) Poincar\'e duality the left 
(resp. right) hand side vanishes for 
$j<n-1$ (resp. $j>n-1$). 
We thus obtain the concentration 
\[H^{j}_{c}(Y_{t};\CC)_{\lambda}\simeq 0 \qquad (j\neq n-1).\] 
By comparing the weights of the mixed Hodge 
modules associated to $R\pi_{!}\CC_{Y}$ and 
$R\pi_{*}\CC_{Y}$ we see that the only 
non-trivial cohomology group 
$H^{n-1}_{c}(Y_{t};\CC)_{\lambda}$ has a pure weight $n-1$. 
See also the proof of 
Sabbah \cite[Theorem 13.1]{Sabbah-2}.  
\end{proof}

Now let $X$ be a general variety and for 
some $\varepsilon >0$ consider a family 
$Y\subset B(0; \e )^* \times X$
of subvarieties of $X$ over the 
punctured disk $B(0; \e )^* \subset \CC$.
Let $\pi: B(0; \e )^* \times X \to B(0; \e )^*$ 
be the projection and
for $t\in\CC$ such that $0<|t|<\varepsilon$ set 
$Y_{t}=Y\cap\pi^{-1}(t)\subset X$.
Then we obtain the monodromy automorphisms
\[\Phi_{j}:H^{j}_{c}(Y_{t};\CC)\xrightarrow{\sim}
H^{j}_{c}(Y_{t};\CC) \qquad (j\in\ZZ)\]
for $0<|t|<\varepsilon$.
For $\lambda\in \CC$ let
\[H^{j}_{c}(Y_{t};\CC)_{\lambda} \subset H^{j}_{c}(Y_{t};\CC)\] 
be the generalized eigenspace of $\Phi_{j}$ for the eigenvalue $\lambda$.
Then we have an isomorphism
\[H^{j}\psi_{t,\lambda}(j_{!}R\pi_{!}\CC_{Y})\simeq 
H^{j}_{c}(Y_{t};\CC)_{\lambda}\]
for any $j \in \ZZ$ and $\lambda\in\CC$.
Considering the mixed Hodge module over 
the vector space 
$H^{j}\psi_{t}(j_{!}R\pi_{!}\CC_{Y})$ we obtain a 
 mixed Hodge structure which 
coincides with the classical limit 
mixed Hodge structure (see El Zein~\cite{E-Z} 
and Steenbrink-Zucker~\cite{S-Z}). 
We denote it by $H^{j}_{c}(Y_{\infty};\CC)$. 
For $\lambda \in \CC$ we denote also by 
$H^{j}_{c}(Y_{\infty};\CC)_{\lambda}$ the 
generalized eigenspace of the monodromy 
\[ \Phi_j: H^{j}_{c}(Y_{\infty};\CC)
\xrightarrow{\sim} H^{j}_{c}(Y_{\infty};\CC) \]
for the eigenvalue $\lambda$ 
endowed with the weight and Hodge filtrations 
induced from those of $H^{j}_{c}(Y_{\infty};\CC)$. 
The weight filtration $M_{\bullet}$ on  
$H^{j}_{c}(Y_{\infty};\CC)$ is the 
``relative'' monodromy filtration 
with respect to its Deligne's weight filtration 
$W_{\bullet}$ in the following sense. 
Let $\Phi_{j}^{u}$ be the unipotent part of the 
monodromy $\Phi_{j}$ and set 
$N=\log{\Phi_{j}^{u}}:H^{j}_{c}(Y_{\infty};\CC)\to 
H^{j}_{c}(Y_{\infty};\CC)$.
Then for any $r\in \ZZ$ the filtration 
$M(r)_{\bullet}$ on the graded piece 
$V_{r}=Gr^{W}_{r}H^{j}_{c}(Y_{\infty};\CC)$ induced by 
$M_{\bullet}$ and the morphism 
$N(r): V_{r}\to V_{r}$ induced by $N$ give rise to isomorphisms 
\[N(r)^{k}:Gr_{r+k}^{M(r)}V_{r}\xrightarrow{\sim} 
Gr^{M(r)}_{r-k}V_{r} \qquad (k\geq 0).\]
Namely the filtration $M(r)_{\bullet}$ on $V_{r}$ is 
the monodromy filtration of 
the automorphism $Gr_{r}^{W}(\Phi_{j}):
V_{r}\xrightarrow{\sim}V_{r}$ centered at $r$. 
For $\lambda\in \CC$ and $p,q,r\in \ZZ$ let 
$h^{p,q}(Gr_{r}^{W}H^{j}_{c}(Y_{\infty};\CC)_{\lambda})\geq 0$
be the dimension of the $(p,q)$-part of 
the above limit mixed Hodge structure on 
the graded piece $Gr^{W}_{r}H^{j}_{c}(Y_{\infty};\CC)_{\lambda}$ 
defined by the weight filtration $M(r)_{\bullet}$.

\begin{definition}[{Stapledon~\cite{Stapledon}}]\label{Epolydef}
For $\lambda\in\CC$ we define the equivariant refined 
limit mixed Hodge polynomial (resp. 
the equivariant limit mixed Hodge polynomial)
$E_{\lambda}(Y_{\infty};u,v,w)\in\ZZ[u,v,w]$ (resp. 
$E_{\lambda}(Y_{\infty};u,v)\in\ZZ[u,v]$) 
for the eigenvalue $\lambda\in\CC$ by 
\begin{align*}
E_{\lambda}(Y_{\infty};u,v,w)&=\sum_{p,q\in \ZZ}\sum_{j\in\ZZ}(-1)^{j}
h^{p,q}(Gr^{W}_{r}H^{j}_{c}(Y_{\infty};\CC)_{\lambda}) \ u^{p}v^{q}w^{r},
\\E_{\lambda}(Y_{\infty};u,v)&=\sum_{p,q\in \ZZ}\sum_{j\in\ZZ}(-1)^{j}
h^{p,q}(H^{j}_{c}(Y_{\infty};\CC)_{\lambda}) \ u^{p}v^{q}.
\end{align*}
\end{definition}
By this definition, obviously we have 
$E_{\lambda}(Y_{\infty};u,v,1)=E_{\lambda}(Y_{\infty};u,v)$ 
for any $\lambda\in\CC$. 

\begin{lemma}\label{new-lemma}
Let $Z \subset Y \subset B(0; \e )^* \times X$ be a 
subfamily of $Y$. Then for any $\lambda \in \CC$ 
we have 
\[
E_{\lambda}(Y_{\infty};u,v,w)=E_{\lambda}(Z_{\infty};u,v,w)+
E_{\lambda}((Y \setminus Z)_{\infty};u,v,w) \qquad 
(0<|t| < \e ). 
\]
\end{lemma}
\begin{proof}
There exists a long exact sequence 
\[
\cdots \to H^j R \pi_!  \QQ^{H}_{Y \setminus Z} \to 
H^j R \pi_!  \QQ^{H}_{Y} \to H^j R \pi_!  \QQ^{H}_{Z} \to 
H^{j+1} R \pi_!  \QQ^{H}_{Y \setminus Z} \to \cdots 
\]
of mixed Hodge modules. For any $r \in \ZZ$ 
by taking the $r$-th graded piece $Gr^{W}_{r}( \cdot )$ 
of each term in it, we obtain again a long 
exact sequence. Then the assertion follows by 
applying the (exact) nearby cycle functor 
$\psi_t( \cdot )$ of mixed Hodge modules 
to them. 
\end{proof}

Now let us return to the family 
$Y \subset B(0; \e )^* \times (\CC^*)^{n}$ of 
hypersurfaces of $T_{0}=(\CC^*)^{n}$ over the punctured disk
$B(0; \e )^*$. 
Then by Proposition \ref{GLMCD} and 
Theorem~\ref{RCS} we obtain the following corollary. 
\begin{corollary}\label{cor:1}
Assume that the family 
$Y \subset B(0; \e )^* \times (\CC^*)^{n}$ is sch\"on.
Then we have
\begin{align*}
E_{\lambda}(Y_{\infty};u,v)&=\sum_{\relint{F}\subset\Int{P}}
E_{\lambda}(V_{F}\circlearrowleft \MH ; u,v)\cdot (1-uv)^{n-\dim{F}},
\end{align*}
for any $\lambda\in\CC$.
Here the equivariant mixed Hodge polynomials 
$E_{\lambda}(V_{F}\circlearrowleft \MH ; u,v)\in\ZZ[u,v]$ 
are defined by Deligne's mixed 
Hodge structure of the variety $V_{F}$ and the 
semisimple action on its cohomology groups as
\begin{align*}
E_{\lambda}(V_{F}\circlearrowleft \MH ; u,v)&=
\sum_{p,q\in \ZZ} \sum_{j\in\ZZ}
(-1)^{j}h^{p,q}(H^{j}_{c}(V_{F};\CC)_{\lambda})u^{p}v^{q}.
\end{align*}
\end{corollary}

The following fundamental 
result was obtained by Stapledon in \cite{Stapledon}. 
For $\lambda \in \CC$ set
\[
\varepsilon ( \lambda ) =\left\{
\begin{array}{ll}
1 &( \lambda =1 )\\\\ 
0 &( \lambda \not= 1 ) \\
\end{array}
\right.
\]
and recall that we have 
\begin{equation}
h^{*}_{\lambda}(P,\nu_{f};u,v,w)= 
\sum_{Q\prec P}w^{\dim{Q}+1}l^{*}_{\lambda}
(Q,\nu_{f}|_{Q};u,v) \cdot g([Q,P];uvw^2).
\end{equation}

\begin{theorem}[{\cite[Theorem 5.7]{Stapledon}}] \label{st:5.7} 
Assume that the family $Y$ is sch\"on.
Then for any $\lambda 
 \in \CC$ we have
\begin{equation}
 uvw^2 E_{\lambda}(Y_{\infty};u,v,w) =
\varepsilon ( \lambda ) \cdot (uvw^2-1)^n + 
(-1)^{n-1} h^{*}_{\lambda}(P,\nu_{f};u,v,w). 
\end{equation}
\end{theorem}

In \cite{K-S-1} and \cite{Stapledon} Katz and 
Stapledon proved this theorem 
from Corollary \ref{cor:1} by 
using \cite[Section 2]{M-T-4} and 
some deep results on combinatorics 
developed by Katz-Stapledon in \cite{K-S-1} that 
build on earlier work of Stanley~\cite{Stanley}. 
By Theorem~\ref{st:5.7} and 
Proposition \ref{Gysin} we 
immediately obtain the following result. 
Note that by the definition of the integers 
$m_F$ the condition $\lambda\notin R_{f}$ 
implies the vanishing $l^{*}_{\lambda}
(Q,\nu_{f}|_{Q};u,v)=0$ for 
any proper face $Q \not= P$ of $P$. 

\begin{theorem}\label{EQEEP} 
Assume that the family $Y$ is sch\"on.
Then for any $\lambda\notin R_{f}$ the equivariant 
mixed Hodge polynomial $E_{\lambda}(Y_{\infty};u,v,w)\in\ZZ[u,v,w]$
for the eigenvalue $\lambda$ is concentrated in degree $n-1$ 
in the variable $w$ and given by
\begin{align*}
E_{\lambda}(Y_{\infty};u,v,w)&=(-1)^{n-1}w^{n-1}\sum_{p,q}h^{p,q}(
H^{n-1}_{c}(Y_{\infty};\CC)_{\lambda})u^{p}v^{q}\\
&=(-1)^{n-1}\frac{w^{n-1}}{uv}l^{*}_{\lambda}(P,\nu_{f};u,v)\\
&=(-1)^{n-1}\frac{w^{n-1}}{uv}\sum_{F\in\mathcal{S}}
v^{\dim{F}+1}l^{*}_{\lambda}(F,\nu_{f}|_{F};uv^{-1})
 \cdot l_{P}(\mathcal{S},F;uv).
\end{align*}
In particular, by setting $u=v=s$ and $w=1$ we have
\begin{align*}
E_{\lambda}(Y_{\infty};s,s)&=(-1)^{n-1}\sum_{k\geq 0} 
\Bigl(
\sum_{p+q=k}h^{p,q}(H^{n-1}_{c}(Y_{\infty};\CC)_{\lambda})
\Bigr) s^{k}\\
&=(-1)^{n-1}\frac{1}{s^2}\sum_{F\in\mathcal{S}}
s^{\dim{F}+1}l^{*}_{\lambda}(F,\nu_{f}|_{F};1)
 \cdot l_{P}(\mathcal{S},F;s^2).
\end{align*}
\end{theorem}

Note that the concentration in this theorem 
corresponds to that in Corollary \ref{CoEQEEP}. 
We also obtain the following corollary. 
Note that for $\lambda\notin R_{f}$ 
by Corollary \ref{CoEQEEP} 
(or Theorem \ref{EQEEP}) and 
the construction of the weight filtration of 
the limit mixed Hodge structure 
$H^{n-1}_{c}(Y_{\infty};\CC)$ 
the filtration on $H^{n-1}_{c}(Y_{\infty};\CC)_{\lambda}$ 
induced by it is equal to the monodromy 
filtration for the monodromy 
$\Phi_{n-1}:H^{n-1}_{c}(Y_{\infty};\CC)_{\lambda}
\xrightarrow{\sim} H^{n-1}_{c}(Y_{\infty};\CC)_{\lambda}$
centered at $n-1$. 

\begin{corollary}\label{th:3}
Assume that the family 
$Y \subset B(0; \e )^* \times (\CC^*)^{n}$ is sch\"on.
Then for any $\lambda\notin R_{f}$ we have the symmetry
\[\sum_{p+q=n-1+k}h^{p,q}(H^{n-1}_{c}(Y_{\infty};\CC)_{\lambda})=
\sum_{p+q=n-1-k}h^{p,q}(H^{n-1}_{c}(Y_{\infty};\CC)_{\lambda})\]
for any $k\geq 0$.
\end{corollary}

By Theorem \ref{EQEEP} and 
Corollary \ref{th:3}, for any $\lambda\notin R_{f}$ the Jordan normal 
form of the middle-dimensional monodromy 
\[\Phi_{n-1}:H^{n-1}_{c}(Y_{t};\CC)_{\lambda}
\xrightarrow{\sim} H^{n-1}_{c}(Y_{t};\CC)_{\lambda}\]
on $H^{n-1}_{c}(Y_{t};\CC)_{\lambda}$ 
can be recovered from the polynomial $E_{\lambda}(
Y_{\infty};u,v,w)\in\ZZ[u,v,w]$ as follows.
By  Theorem \ref{EQEEP} for the polynomial $E_{\lambda}(
Y_{\infty};u,v)=E_{\lambda}(Y_{\infty};u,v,1)\in\ZZ[u,v]$ we have
\[E_{\lambda}(Y_{\infty};u,v,w)=E_{\lambda}(Y_{\infty};u,v)\cdot w^{n-1}.\]
Moreover the 
polynomial $\tl{E_{\lambda}}(Y_{\infty};s):=
(-1)^{n-1}E_{\lambda}(Y_{\infty};s,s)\in\ZZ[s]$ has 
only non-negative coefficients and the symmetry centered at $n-1$.
By the Lefschetz decomposition of 
$H^{n-1}_{c}(Y_{\infty};\CC)_{\lambda}$ 
there exist non-negative integers 
$q_{\lambda,i}\geq 0~(0\leq i \leq n-1)$ such that 
\begin{align*}
\tl{E_{\lambda}}(Y_{\infty};s)=
&q_{\lambda,0}(1+s^{2}+ \cdots +s^{2n-4}+s^{2n-2})
\\&+q_{\lambda,1}(s+s^{3}+ \cdots +s^{2n-3})
\\&+q_{\lambda,2}(s^{2}+ \cdots +s^{2n-4})
\\&+ \cdots \cdots 
\\&+q_{\lambda,n-1}s^{n-1}.
\end{align*}
For $\lambda\in \CC$ and $m\geq 1$ denote by $J_{\lambda,m}$ the 
number of the Jordan blocks in the monodromy automorphism
\[\Phi_{n-1}:H^{n-1}_{c}(Y_{t};\CC)\xrightarrow{\sim}H^{n-1}_{c}(Y_{t};\CC) 
\qquad (0<|t|\ll1)\]
for the eigenvalue $\lambda$ with size $m$.

\begin{proposition}\label{th:4}
Assume that the family $Y$ is sch\"on.
Then for any $\lambda\notin R_{f}$ we have
\[J_{\lambda,m}=q_{\lambda,n-m} \qquad (1\leq m\leq n).\]
\end{proposition}

Recall that for a cell $F\in\mathcal{S}$ the local 
$h$-polynomial $l_{P}(\mathcal{S},F;t)\in\ZZ[t]$ 
has non-negative coefficients and the symmetry
\[l_{P}(\mathcal{S},F;t)=t^{n-\dim{F}}l_{P}(\mathcal{S},F;t^{-1})\]
(see \cite[Remark 4.9]{Stapledon}).
Moreover it is unimodal.
Hence there exist non-negative integers $l_{F,i}~(0\leq i 
\leq \lfloor \frac{n-\dim{F}}{2} \rfloor)$ 
such that
\begin{align*}
l_{P}(\mathcal{S},F;t)&=l_{F,0}(1+t+t^2+ \cdots+t^{n-\dim{F}})
\\&+l_{F,1}(t+t^2+ \cdots +t^{n-\dim{F}-1})
\\&+l_{F,2}(t^2+ \cdots +t^{n-\dim{F}-2})
\\&+ \cdots \cdots.
\end{align*}
We set
\[\tl{l}_{P}(\mathcal{S},F;t)=
\sum_{i=0}^{\lfloor\frac{n-\dim{F}}{2}\rfloor}l_{F,i}t^{i}.\]
Then by Theorem \ref{EQEEP} and 
Proposition \ref{th:4} we obtain the following result. 

\begin{theorem}\label{torusjordan}
Assume that the family $Y$ is sch\"on. 
Then for $\lambda\notin R_{f}$ we have
\[\sum_{m=0}^{n-1}J_{\lambda,n-m}s^{m+2}=\sum_{F\in\mathcal{S}}
s^{\dim{F}+1}l^{*}_{\lambda}(F,\nu_{f}|_{F};1) 
  \cdot \tl{l}_{P}(
\mathcal{S},F;s^{2}).\]
In particular, we have
\[J_{\lambda,n}=\sum_{F\in\mathcal{S},\ \dim{F}=1}
l^{*}_{\lambda}(F,\nu_{f}|_{F};1) \cdot l_{F,0}.\] 
\end{theorem}

The multiplicities of the eigenvalues 
$\lambda \not= 1$ in the 
middle-dimensional monodromy 
$\Phi_{n-1}$ are described more simply as follows. 

\begin{theorem}\label{torusmulti}
Assume that the family $Y$ is sch\"on. Then 
for $\lambda \not= 1$ 
the multiplicity of the factor $t- \lambda$ in 
the characteristic polynomial of the monodromy 
\[\Phi_{n-1}:H^{n-1}_{c}(Y_{t};\CC)\xrightarrow{\sim}H^{n-1}_{c}(Y_{t};\CC) 
\qquad (0<|t|\ll1)\]
is equal to that in 
\[ 
\prod_{\relint F \subset \Int P, \ \dim F=n} 
(t^{m_F}-1)^{\Vol_{\ZZ}( \widetilde{F} )}, 
\]
where $\Vol_{\ZZ}( \widetilde{F}) \in \ZZ_{>0}$ 
is the normalized volume i.e. the 
$n!$ times usual volume $\Vol ( \widetilde{F})$ 
of $\widetilde{F}$ 
with respect to the lattice $\Aff ( \widetilde{F} ) 
\cap \ZZ^{n+1} \simeq \ZZ^{n}$ in 
$\Aff ( \widetilde{F} ) \simeq \RR^{n}$. 
\end{theorem}

\begin{proof}
By Proposition \ref{Gysin} and the proof of Theorem \ref{RCS}, 
the assertion can be proved by calculating 
monodromy zeta functions as in 
\cite{M-T-2}. We can obtain it also just by 
taking the Euler characteristics of the 
both sides of the equality in Theorem \ref{RCS}. 
\end{proof}

\section{ Monodromies and limit mixed Hodge structures of families of 
hypersurfaces in $\CC^n$}\label{sec:4}

By Corollary~6.3 of \cite{Stapledon}, 
we can describe the Jordan normal forms of the 
monodromy automorphisms on the cohomology groups 
of a family of hypersurfaces in $\CC^n$ in 
terms of the polynomials defined in 
Subsection~\ref{polysec} under some conditions 
(see Remark~\ref{StaCom}).
In this section, we show that the monodromies 
for some eigenvalues can be described even 
if we drop the conditions imposed in 
Corollary~6.3 of \cite{Stapledon}.

Let $f(t,x)=\sum_{v\in\ZZ^n_{+}}a_{v}(t)x^{v}\in\BK[x_1,
\dots,x_n]~(a_{v}(x)\in\BK)$
be a polynomial of $x=(x_{1},\dots,x_{n})$ 
over the field $\BK=\CC(t)$ of rational functions of $t$.
Then as in Section~\ref{sec:3} we can define an (unbounded) polyhedron 
$\UH_{f}$ associated to it in $\RR^n_{+}\times \RR^{1}$ and 
its projection $P=p(\UH_{f})\subset \RR^{n}_{+}$.
We call $P$ the Newton polytope of $f\in \BK[x_{1},\dots,x_{n}]$.
Throughout this section we assume that $\dim P=n$.

Let $\Sigma_{0}$ be the dual fan of 
$\UH_{f}$ in $\RR^n\times\RR^{1}_{+}
\subset \RR^{n+1}$ and $\nu_{f}:P\to\RR$ the function 
defining the bottom part of the boundary
$\partial \UH_{f}$ of $\UH_{f}$.
Moreover by the subdivision $\mathcal{S}$ of $P$ into the lattice 
polytopes $p(\tl{F})~(\tl{F}\prec \UH_{f})$ we define polynomials
$I_{f}^{F}(x)\in\CC[x_{1},\dots ,x_{n}]$ and elements 
$[V_{F}\circlearrowleft \MH] 
\in \M_{\CC}^{\hat{\mu}}$ for cells $F \in \mathcal{S}$ 
as in Section \ref{sec:3}. 
In this situation, the hypersurface 
 $f^{-1}(0)\subset X= \CS_{t}\times{\CC}^{n}_{x}$ 
defines a family $Y$ of hypersurfaces of $X_{0}={\CC}^{n}_{x}$
over a small punctured disk $B(0; \e )^*$ ($0<\varepsilon\ll1$).
By the projection $\pi:\CS_{t}\times {\CC}^{n}_{x}
\twoheadrightarrow\CS_{t}$ for 
$t\in \CC$ such that $0<|t|<\varepsilon$ we set
$Y_{t}:=\pi^{-1}(t)\cap Y \subset \{ t 
\}\times X_{0}\simeq X_{0}={\CC}^{n}_{x}$.
We define also the sch\"onness of the family as in 
Definition~\ref{schon}.

We set $P_{\infty}:
=\overline{\partial P\cap \INT{\RR^{n}_{+}}} 
\subset \partial P$ and 
define a finite subset $R_{f}\subset \CC$ by 
\[R_{f}=\bigcup_{F\subset P_{\infty}}\{\lambda\in\CC\ |\ 
\lambda^{m_{F}}=1\}\subset \CC.\] 
Then we have the following result.

\begin{theorem}\label{th:syucyuu}
Assume that the family $Y$ of hypersurfaces 
in $X_{0}=\CC^n$ is sch\"on. 
Then for $\lambda\notin R_{f}$ the 
equivariant refined limit 
mixed Hodge polynomial $E_{\lambda}(Y_{\infty};u,v,w)\in\ZZ[u,v,w]$
for the eigenvalue $\lambda$ is concentrated in 
degree $n-1$ in the variable $w$ and given by
\begin{align*}
E_{\lambda}(Y_{\infty};u,v,w)&=(-1)^{n-1}w^{n-1}
\sum_{p,q}h^{p,q}(H^{n-1}_{c}(Y_{\infty};\CC)_{\lambda})u^{p}v^{q}\\
&=(-1)^{n-1}\frac{w^{n-1}}{uv}l^{*}_{\lambda}(P,\nu_{f};u,v)\\
&=(-1)^{n-1}\frac{w^{n-1}}{uv}\sum_{F
\in\mathcal{S}}v^{\dim{F}+1}l^{*}_{\lambda}(F,
\nu_{f}|_{F};uv^{-1}) \cdot l_{P}(\mathcal{S},F;uv).
\end{align*}
\end{theorem}

\begin{proof}
For a possibly empty subset $I\subset \{1,\dots,n\}$,
we define a subset $T^{I}$ of $X_{0}=\CC^n$ by
\[T^{I}:=\{(x_{1},\dots,x_{n})\in X_{0} \ | \ 
x_{i}=0~(i\notin I), x_{i}\neq 0~(i\in I)\}\simeq(\CS)^{|I|}.\]
Then we have a decomposition $X_{0}= \CC^n = \bigsqcup_{
I\subset\{1,\dots,n\}}T^{I}$ of $X_{0}= \CC^n$. 
We also define a polynomial $f_{I}\in \BK[(x_{i})_{i\in I}]$ 
by substituting $0$ into the variable $x_{i}~(i\notin I)$ of $f$, 
a family of hypersurfaces $Y^{I}$ of $T^{I}$ by 
$Y^{I}:=f^{-1}_{I}(0) 
\subset B^* \times T^I$ 
and a polytope $P^{I}$ in 
$\RR^{I}=\{(x_{1},\dots,x_{n}) \in \RR^n \mid x_{i}=0~(i 
\notin I)\}
\subset \RR^n$ by $P^{I}:=P \cap \RR^{I}$. 
Then by Lemma \ref{new-lemma} we have 
\[
E_{\lambda}(Y_{\infty};u,v,w)= 
\sum_{I\subset \{1,\dots,n\}} E_{\lambda}(Y_{\infty}^{I};u,v,w).
\]
We shall say that a face $Q \prec P$ of $P$ 
is relevant if $Q \not\subset P_{\infty}$. 
If $Q \prec P$ is relevant, then for 
any face $\sigma$ of the first quadrant  
$\RR_+^n$ containing $Q$ 
the face $\sigma \cap P \prec P$ 
of $P$ is also relevant. 
Moreover there exist a possibly empty 
subset $I \subset \{1,\dots, n\}$  
such that $Q=P^I$ and $\dim P^I = 
|I|$. We denote 
by $S$ the set consisting of 
possibly empty subsets $I \subset\{1,\dots, n \}$ 
such that $P^I$ are relevant. 
Then by Theorem \ref{st:5.7} 
for $\lambda\notin R_{f}$ we have 
\[ 
uvw^{2} E_{\lambda}(Y_{\infty};u,v,w)= 
\sum_{I \in S}
(-1)^{|I|-1} 
h^{*}_{\lambda}(P^{I},\nu_{f}|_{P^{I}} ;u,v,w).
\] 
Moreover for each relevant face 
$P^I \prec P$ of $P$ 
by Definition~\ref{def:poly} we have
\begin{align}\label{cieq33}
h^{*}_{\lambda}(P^{I},\nu_{f}|_{P^{I}} ;u,v,w) 
=\sum_{Q \prec P^{I}}
w^{\dim Q+1}l^{*}_{\lambda}(Q, \nu_{f}|_{Q};u,v) 
\cdot g([Q,P^{I}];uvw^2).
\end{align}
If $Q\prec P^{I}$ is not a 
relevant face of $P$, then $Q \subset P_{\infty}$ and 
for $\lambda \notin R_{f}$ 
we have 
$l^{*}_{\lambda}(Q, \nu_{f}|_{Q};u,v) =0$. Note also 
that $l^{*}_{\lambda}(\emptyset, \nu_{f}|_{\emptyset};u,v)=0$ 
for $\lambda\neq1$. 
Moreover for any $I, I' \in S$ such that 
$I'\subset I$ we have $g([P^{I'},P^{I}];uvw^2)=1$.  
Hence for each fixed $I' \in S$ we have 
\[
\sum_{I: I'\subset I}(-1)^{|I|-1} g([P^{I'},P^{I}];uvw^2) 
= \sum_{I: I'\subset I}(-1)^{|I|-1} = 
\left\{
\begin{array}{ll}
(-1)^{n-1}&(I'=\{1,\dots,n\})\\\\ 
0&(otherwise). \\
\end{array}
\right.
\]
We thus obtain 
\begin{align*}
uvw^2E_{\lambda}(Y_{\infty};u,v,w)&=\sum_{I \in S}
(-1)^{|I|-1}{\sum_{I' \in S: I'\subset I}w^{|I'|+1}l^{*}_{
\lambda}(P^{I'},\nu_{f}|_{P^{I'}};u,v)} 
\cdot g([P^{I'},P^{I}];uvw^2) \\
&=\sum_{I' \in S}{w^{|I'|+1}l^{*}_{\lambda}(
P^{I'},\nu_{f}|_{P^{I'}};u,v)
\cdot \Bigl\{ \sum_{I: I'\subset I}(-1)^{|I|-1} 
g([P^{I'},P^{I}];uvw^2) \Bigr\}} \\
&=(-1)^{n-1}w^{n+1}l^{*}_{\lambda}(P,\nu_{f};u,v).
\end{align*}
\end{proof}

We shall say that a face $\sigma \prec \RR_+^n$ of 
the first quadrant $\RR_+^n$ is relevant if 
the condition $(P \setminus P_{\infty}) \cap \sigma 
\not= \emptyset$ is satisfied. It is easy to see 
that if $\sigma \prec \RR_+^n$ is relevant 
then we have $\dim (P \cap \sigma )= \dim \sigma$. 
Let $\Sigma_1$ 
be the fan in $\RR^n$ consisting of all the faces of 
$\RR_+^n$ and regard it 
as the dual fan of the first quadrant $\RR_+^n$. 
Denote by $\Sigma_1^{\circ} \subset \Sigma_1$ its subset 
consisting of the dual cones of 
the relevant faces of $\RR_+^n$. Then we 
can easily see that $\Sigma_1^{\circ}$ is a 
subfan of $\Sigma_1$. Denote by $\Omega_0$ 
the toric variety associated to $\Sigma_1^{\circ}$. 
Then $\Omega_0$ is an open subset of $X_0= \CC^n$ 
and $X_0 \setminus \Omega_0$ is a closed subset 
in it. Moreover for the action of $T_0=( \CC^*)^n$ 
on $X_0= \CC^n$ it is a union of some $T_0$-orbits. 
Set $Y^{\circ}=Y \cap ( \CC^* \times 
\Omega_0) \subset \CC^* \times \Omega_0$ 
and let $\pi^{\circ}: 
\CC^* \times \Omega_0 \to \CC^*$ be the 
projection. 

\begin{theorem}\label{th:5}
Assume that the family $Y$ of hypersurfaces in $X_{0}=\CC^{n}$ is sch\"on. 
Then for any $\lambda\notin R_{f}$ the morphism
\[\psi_{t,\lambda}(j_{!}R\pi_{!}\CC_{Y^{\circ}}) \longrightarrow 
\psi_{t,\lambda}(j_{!}R\pi_{!}\CC_{Y})\]
induced by the one $\CC_{Y^{\circ}} \to \CC_{Y}$ is an isomorphism. 
Moreover for such $\lambda$ the morphism 
\[\psi_{t,\lambda}(j_{!}R (\pi^{\circ})_{!}\CC_{Y^{\circ}}) \longrightarrow 
\psi_{t,\lambda}(j_{!}R( \pi^{\circ})_{*}\CC_{Y^{\circ}})\]
induced by the one $R (\pi^{\circ})_{!}\CC_{Y^{\circ}}\to 
R ( \pi^{\circ})_{*}\CC_{Y^{\circ}}$ is an isomorphism.
\end{theorem}

\begin{proof}
The proof is similar to that of Theorem~\ref{th:2}. 
By decomposing 
the closed subset $X_0 \setminus \Omega_0$ 
into tori and applying Proposition \ref{Gysin} 
and Theorem \ref{EQEEP} to each of them 
(see Lemma \ref{smooth}), for 
$\lambda\notin R_{f}$ 
we obtain the vanishing  
\[\psi_{t,\lambda}(j_{!}R\pi_{!}\CC_{Y \setminus 
Y^{\circ}}) \simeq 0 \]
from which the first assertion follows. 
Let $\Xi_0$ be the subfan of the dual fan $\Sigma_0$ 
in $\RR^n \simeq \RR^{n} 
\times \{0\}$ consisting of the 
cones $\sigma \in \Sigma_0$
contained in $\RR^{n}\times \{0\}\subset \RR^{n+1}$. 
Then $\Xi_0$ is the dual fan of the $n$-dimensional 
polytope $P \subset \RR^n$. Moreover by the 
definition of $\Sigma_1^{\circ}$ we can easily 
see that $\Sigma_1^{\circ}$ is a subfan of 
$\Xi_0$. By this property we 
can construct a smooth subdivision $\Sigma$ of 
$\Sigma_0$ such that $\Sigma_1^{\circ} 
\subset \Sigma$. Then the toric variety 
$X_{\Sigma}$ associated to $\Sigma$ 
is a smooth variety containing 
$\CC^* \times \Omega_0$ and the second 
assertion can be proved as in the proof 
of Theorem~\ref{th:2}. 
\end{proof}

For $\lambda\in \CC$ let
\[H^{j}_{c}(Y_{t};\CC)_{\lambda} \subset H^{j}_{c}(Y_{t};\CC)\] 
be the generalized eigenspace of 
$\Phi_{j}$ for the eigenvalue $\lambda$. 
Similarly for $Y^{\circ} \subset Y$ 
we define linear subspaces 
\[H^{j}_{c}(Y^{\circ}_{t};\CC)_{\lambda} \subset 
H^{j}_{c}(Y^{\circ}_{t};\CC) 
\qquad ( \lambda \in \CC ).\] 
Then by Theorem \ref{th:5} for any $\lambda\notin R_{f}$ 
there exist isomorphisms  
\[H^{j}_{c}(Y^{\circ}_{t};\CC)_{\lambda} 
\simeq H^{j}_{c}(Y_{t};\CC)_{\lambda} \qquad (j \in \ZZ). \]

\begin{corollary}\label{Ncc}
Assume that the family $Y$ of hypersurfaces 
in $X_{0}=\CC^{n}$ is sch\"on. 
Then for any $\lambda\notin R_{f}$ 
and $t\in \CS$ such that $0<|t|\ll1$ we have the 
concentration 
\[H^{j}_{c}(Y_{t};\CC)_{\lambda} 
\simeq 0 \qquad (j \not= n-1)\] 
and the filtration on the only non-trivial 
cohomology group $H^{n-1}_c(Y_{t};\CC)_{\lambda}$ 
induced by Deligne's weight filtration on 
$H^{n-1}_c(Y_{t};\CC)$ is concentrated in 
degree $n-1$. 
\end{corollary}

\begin{proof}
For $\lambda\in \CC$ and $j \in \ZZ$ let
\[H^{j}(Y^{\circ}_{t};\CC)_{\lambda} 
\subset H^{j}(Y^{\circ}_{t};\CC)\] 
be the generalized eigenspace of the monodromy 
$H^{j}(Y^{\circ}_{t};\CC)\xrightarrow{\sim}
H^{j}(Y^{\circ}_{t};\CC)$. Then by Theorem \ref{th:5} 
for any $\lambda\notin R_{f}$ 
and $t\in \CS$ such that $0<|t|\ll1$ we have 
isomorphisms 
\[H^{j}_{c}(Y^{\circ}_{t};\CC)_{\lambda} \simeq 
H^{j}(Y^{\circ}_{t};\CC)_{\lambda} \qquad (j \in \ZZ ).
\]
In the same way as in the proof 
of Corollary~\ref{CoEQEEP},
we see that the filtration on 
$H^{j}_{c}(Y^{\circ}_{t};\CC)_{\lambda}$ 
induced by Deligne's weight filtration of 
$H^{j}_{c}(Y^{\circ}_{t};\CC)$ is concentrated in degree $j$.
Then, the assertion follows immediately from Theorem~\ref{th:syucyuu}.
\end{proof}

\begin{remark}\label{imrem} 
By the proofs of Theorems \ref{th:2} and \ref{th:5}, 
if the family $Y$ is sch\"on 
we can also show that for any $\lambda\notin R_{f}$ 
and $t\in \CS$ such that $0<|t|\ll1$ 
there exist isomorphisms 
\[H^{j}(Y_{t};\CC)_{\lambda} \simeq 
H^{j}(Y^{\circ}_{t};\CC)_{\lambda} \qquad (j \in \ZZ ).
\]
Indeed, for $Z= \CC^* \times (X_0 \setminus \Omega_0) \subset X$ 
it suffices to show the vanishing 
\[\psi_{t,\lambda}(j_{!}R\pi_{*} R \Gamma_Z( \CC_{Y})) \simeq 0. \]
With the help of Lemma \ref{smooth} we can show it 
by decomposing $Z$ into tori as in the proof of Theorem \ref{th:5}. 
\end{remark}

\begin{remark}
Note that $Y_{t} \subset \CC^n$ may not be smooth in our situation.
Therefore, we can not deduce the concentration 
$H^{j}_{c}(Y_{t};\CC)\simeq 0$ $(j\neq n-1)$ in a 
similar way as in the proof of Corollary~\ref{CoEQEEP}. 
Moreover $Y_t^{\circ}$ may not be affine. For these reasons, 
we relied on Theorem~\ref{th:syucyuu} in the proof of Corollary~\ref{Ncc}.
\end{remark}

We can prove the 
following formula for the 
multiplicities of the eigenvalues 
$\lambda\notin R_{f}$ in the monodromy 
$\Phi_{n-1}$ by calculating 
monodromy zeta functions as in 
\cite{M-T-2}. For a cell $F \in \mathcal{S}$ 
denote by $Q_F \prec P$ the unique 
face of $P$ such that $\relint F \subset \relint Q_F$. 

\begin{theorem}\label{affinemulti}
Assume that the family $Y$ of hypersurfaces 
in $X_{0}=\CC^n$ is sch\"on. Then 
for $\lambda\notin R_{f}$ 
the multiplicity of the factor $t- \lambda$ in 
the characteristic polynomial of the monodromy 
\[\Phi_{n-1}:H^{n-1}_{c}(Y_{t};\CC)
\xrightarrow{\sim}H^{n-1}_{c}(Y_{t};\CC) 
\qquad (0<|t|\ll1)\]
is equal to that in 
\[ 
\prod_{F \not\subset P_{\infty}, \ \dim F= \dim Q_F} 
(t^{m_F}-1)^{(-1)^{n- \dim F} \Vol_{\ZZ}( \widetilde{F} )}, 
\]
where $\Vol_{\ZZ}( \widetilde{F}) \in \ZZ_{>0}$ 
is the normalized volume of $\widetilde{F}$ 
with respect to the lattice $\Aff ( \widetilde{F} ) 
\cap \ZZ^{n+1} \simeq \ZZ^{\dim F}$ in 
$\Aff ( \widetilde{F} ) \simeq \RR^{\dim F}$. 
\end{theorem}

Moreover by Theorems \ref{th:syucyuu} and \ref{th:5} 
and Corollary \ref{Ncc} 
we can easily obtain results 
similar to the ones in Corollary 
\ref{th:3}, Proposition \ref{th:4} and Theorem \ref{torusjordan}. 
In particular as in Theorem \ref{torusjordan},   
by Corollary \ref{Ncc} and Theorem \ref{th:syucyuu} we can 
describe the numbers $J_{\lambda,m}$ of the 
Jordan blocks in the middle-dimensional monodromy 
\[\Phi_{n-1}:H^{n-1}_{c}(Y_{t};\CC)\xrightarrow{\sim}
H^{n-1}_{c}(Y_{t};\CC) \qquad (0<|t|\ll 1)\]
for the eigenvalues $\lambda\notin R_{f}$ 
with size $m\geq 0$ in terms of ${\rm UH}_{f}$. 
See Theorem \ref{NJBTH}. 

\begin{remark}\label{StaCom}
One says that $P$ is convenient if $\dim{P\cap H}=
\dim{H}$ for any coordinate subspace $H$ of $\RR^n$.
In Corollary~6.3 of \cite{Stapledon},
the author assumed that $P$ is convenient 
and the function $\nu_{f}$ is constant on $P_{\infty}$.
Under these assumptions, we have $R_{f}=\{1\}$.
Therefore, we can describe the Jordan normal 
form for the eigenvalues $\lambda\neq 1$ in terms 
of the polynomials associated to $\UH_{f}$.
\end{remark}

\section{Monodromies and limit mixed Hodge 
structures of families of complete 
intersection varieties}\label{sec:5}

In this section, we extend our previous results to 
families complete intersection subvarieties in $(\CS)^{n}$ or $\CC^n$.
Throughout this section, for $1\leq k \leq n$ 
let $f_{i}(t,x)~(1\leq i \leq k)$ be Laurent polynomials
$f_{i}(t,x)=\sum_{v\in\ZZ^n}a_{i,v}(t)x^{v}\in
\BK[x^{\pm}_{1},\dots,x^{\pm}_{n}]$
or polynomials $f_{i}(t,x)=\sum_{v\in\in\ZZ^{n}_{+}}
a_{i,v}(t)x^{v}\in \BK[x_{1},\dots,x_{n}]$
over the field $\BK=\CC(t)$.
Then the subvariety $f^{-1}_{1}(0)\cap\dots\cap 
f^{-1}_{k}(0)$ in $T=\CS_{t}\times (\CS)^{n}_{x}$
or $X=\CS_{t}\times \CC^{n}_{x}$ defines a family $Y$ 
of subvarieties of $T_{0}=(\CS)^{n}_{x}$ 
or $X_{0}=\CC^{n}$ over a small punctured disk
$B(0; \e )^* \subset \CC$ ($0<\varepsilon\ll1 $). 
We shall describe its monodromy and limit mixed Hodge structure.
As in Section \ref{sec:3} 
we define $\UH_{f_{i}}\subset\RR^{n+1}=\RR^n\times\RR_{s}$ and their 
second projections $P_{i}=p(\UH_{f_{i}})\subset \RR^{n}$. 
Set $f=(f_{1},\dots,f_{k})$ and let
\[\UH_{f}:=\UH_{f_{1}}+\dots+\UH_{f_{k}}\subset\RR^{n+1}\]
be the Minkowski sum of $\UH_{f_{1}},\dots,\UH_{f_{k}}$.
Set $P=p(\UH_{f})=P_{1}+\dots+ P_{k}\subset\RR^{n}$. 
Throughout this section we assume that $\dim P=n$. 
By using $\UH_{f}$, we define a function $\nu_{f}:P\to\RR$, a subdivision 
$\mathcal{S}$ of $P$ into lattice polytopes 
and a closed subset 
$P_{\infty} \subset P$ as in Sections \ref{sec:3} and \ref{sec:4}. 
Moreover for each cell $F\in \mathcal{S}$ 
and $1\leq i\leq k$ the initial 
Laurent polynomial 
$I^{F}_{f_{i}}(x)\in\CC[x^{\pm}_{1},\dots,x^{\pm}_{n}]$
of $f_{i}$ with respect to $F$ is defined.

\begin{definition}[{Stapledon~\cite{Stapledon}}]\label{cischon}
We say that the family $Y=f^{-1}_{1}(0)\cap\dots\cap f^{-1}_{k}(0)$ of 
subvarieties of $T_{0}=(\CS)^{n}$ or $X_{0}=\CC^n$ is sch\"{o}n
if for any $J\subset \{1,\dots,k\}$ and any cell $F\in \mathcal{S}$ the 
subvariety $V_{F}=\bigcap_{j\in J}\{I^{F}_{j}=0\}
\subset T_{F}$ of $T_{F}\simeq (\CS)^{\dim{F}}$
is a non-degenerate complete intersection (see \cite{Oka}).
\end{definition}

It follows easily from the proof of 
Theorems~\ref{RCS} and \ref{th:2} that if the 
family $Y$ in $T_{0}=(\CS)^{n}$ is sch\"{o}n its generic fiber 
$Y_{t}=Y\cap\pi^{-1}(t) \subset T_0 \ (0<|t|\ll1)$ 
is a smooth complete intersection. 
Moreover by Danilov-Khovanskii~\cite[Theorem~6.4]{D-K} 
we obtain the following results. 
For $\lambda\in \CC$ and $j \in \ZZ$ let
\[H^{j}_{c}(Y_{t};\CC)_{\lambda} \subset H^{j}_{c}(Y_{t};\CC)\] 
be the generalized eigenspace of the monodromy automorphism 
\[\Phi_{j}:H^{j}_{c}(Y_{t};\CC)\xrightarrow{\sim}
H^{j}_{c}(Y_{t};\CC) \] 
for the eigenvalue $\lambda$. 

\begin{proposition}
Let $Y=f^{-1}_{1}(0)\cap\dots\cap f^{-1}_{k}(0)$ be a 
family of subvarieties in 
$T_{0}=(\CS)^{n}~({\rm resp}. ~X_{0}=\CC^n)$.
Assume that $Y$ is sch\"{o}n and $\dim P_i=n$ 
(resp. $P_i$ is convenient) for any $1 \leq i \leq k$. 
Then for $t\in\CS$ such that $0<|t|\ll1$ we have
\[H^{j}_{c}(Y_{t};\CC)\simeq 0 \qquad (j<n-k)\]
and the Gysin map
\begin{align}
H^{j}_{c}(Y_{t};\CC) \longrightarrow  H^{j+2k}_{c}(T_{0};\CC)
\\ ({\rm resp.} ~H^{j}_{c}(Y_{t};\CC) \longrightarrow  
H^{j+2k}_{c}(X_{0};\CC))
\end{align}
associated to the inclusion map $Y_{t}\hookrightarrow T_{0}~ 
({\rm resp.}~Y_{t}\hookrightarrow X_{0})$
is an isomorphism for $j>n-k$ and surjective for $j=n-k$.
Moreover the monodromy $\Phi_{j}\colon H^{j}_{c}(Y_{t};\CC)
\xrightarrow{\sim}H^{j}_{c}(Y_{t};\CC)$
is identity for $j>n-k$. 
In particular, 
for any $\lambda \not= 1$ and $t\in\CS$ such that $0<|t|\ll1$ 
we have the concentration 
\[H^{j}_{c}(Y_{t};\CC)_{\lambda} \simeq 0  
\qquad (j \not= n-k).\]
\end{proposition}

Defining a finite subset $R_{f}\subset \CC$ by 
using $\partial P, P_{\infty} \subset P$ as in 
Sections \ref{sec:3} and \ref{sec:4}, 
we obtain the following results. 
In the case where $Y$ is family of subvarieties in 
$X_{0}= \CC^n$ we define $\Omega_0 \subset \CC^n$, 
$Y^{\circ} \subset Y$ and the projection 
$\pi^{\circ}: \CC^* \times \Omega_0 \to \CC^*$ 
as in Section \ref{sec:4}. 

\begin{theorem}\label{CIisom} 
Let $Y=f^{-1}_{1}(0)\cap\dots 
\cap f^{-1}_{k}(0)$ be a family of subvarieties in 
$T_{0}=(\CS)^n$. 
Assume that $Y$ is sch\"{o}n.
Then for any $\lambda\notin R_{f}$ the morphism
\[\psi_{t,\lambda}(j_{!}R\pi_{!}\CC_{Y})
\longrightarrow 
\psi_{t,\lambda}(j_{!}R\pi_{*}\CC_{Y})\]
induced by the one $R\pi_{!}\CC_{Y}\to 
R\pi_{*}\CC_{Y}$ is an isomorphism.
\end{theorem} 

\begin{corollary}\label{CINcc}
In the situation of Theorem \ref{CIisom}, 
for any $\lambda\notin R_{f}$ 
and $t\in \CS$ such that $0<|t|\ll1$ we have the 
concentration 
\[H^{j}_{c}(Y_{t};\CC)_{\lambda} 
\simeq 0 \qquad (j \not= n-k)\]
and the filtration on the only non-trivial 
cohomology group $H^{n-k}_c(Y_{t};\CC)_{\lambda}$ 
induced by Deligne's weight filtration on 
$H^{n-k}_c(Y_{t};\CC)$ is concentrated in 
degree $n-k$. 
\end{corollary}

\begin{theorem}\label{th:555}
Let $Y=f^{-1}_{1}(0)\cap\dots 
\cap f^{-1}_{k}(0)$ be a family of subvarieties in 
$X_{0}=\CC^{n}$. 
Assume that $Y$ is sch\"{o}n. 
Then for any $\lambda\notin R_{f}$ the morphism
\[\psi_{t,\lambda}(j_{!}R\pi_{!}\CC_{Y^{\circ}}) \longrightarrow 
\psi_{t,\lambda}(j_{!}R\pi_{!}\CC_{Y})\]
induced by the one $\CC_{Y^{\circ}} \to \CC_{Y}$ is an isomorphism. 
Moreover for such $\lambda$ the morphism 
\[\psi_{t,\lambda}(j_{!}R (\pi^{\circ})_{!}\CC_{Y^{\circ}}) \longrightarrow 
\psi_{t,\lambda}(j_{!}R( \pi^{\circ})_{*}\CC_{Y^{\circ}})\]
induced by the one $R (\pi^{\circ})_{!}\CC_{Y^{\circ}}\to 
R ( \pi^{\circ})_{*}\CC_{Y^{\circ}}$ is an isomorphism.
\end{theorem}

We have also a generalization of Remark \ref{imrem}. 
By Corollary \ref{CINcc} and 
Bernstein-Khovanskii-Kushnirenko's theorem, 
in the case where $Y$ is family of subvarieties in 
$T_{0}=( \CC^*)^n$ we obtain the 
following formula for the 
multiplicities of the eigenvalues 
$\lambda\notin R_{f}$ in the middle-dimensional monodromy
$\Phi_{n-k}\colon H^{n-k}_{c}(Y_{t};\CC) 
\xrightarrow{\sim} H^{n-k}_{c}(Y_{t};\CC)$ 
by calculating monodromy zeta functions as in 
\cite{M-T-2}. 

\begin{definition}\label{Mixvol}
Let $\Delta_1,\ldots,\Delta_n$ 
be lattice 
polytopes in $\RR^n$. Then we define 
their normalized ($n$-dimensional) 
mixed volume 
$\Vol_{\ZZ}( \Delta_1,\ldots,\Delta_n) 
\in \ZZ$ by the formula 
\begin{equation}
\Vol_{\ZZ}( \Delta_1, \ldots , \Delta_n)=
\frac{1}{n!} 
\dsum_{k=1}^n (-1)^{n-k} 
\sum_{\begin{subarray}{c}I\subset 
\{1,\ldots,n\}\\ |I| =k\end{subarray}}
\Vol_{\ZZ}\left(
\dsum_{i\in I} \Delta_i \right)
\end{equation}
where $\Vol_{\ZZ}(\ \cdot\ )
= n! \Vol (\ \cdot\ ) \in \ZZ$ is 
the normalized ($n$-dimensional) volume 
with respect to the lattice $\ZZ^n 
\subset \RR^n$.
\end{definition}

Let $\Sigma_0$ be the dual fan of $\UH_f$ in $\RR^{n+1}$. 
For a cell $F$ in $\mathcal{S}$ let 
$\widetilde{F}\prec \UH_{f}$ be the  
unique compact face 
of $\UH_{f}$ such that $F=p(\widetilde{F})$ 
and $F^{\circ} \in \Sigma_0$ the 
cone which corresponds to it 
in the dual fan $\Sigma_0$. Then for 
the supporting faces $\widetilde{F_i} \prec \UH_{f_i}$ 
of $F^{\circ}$ in $\UH_{f_i}$ 
($1 \leq i \leq k$) we have 
$\widetilde{F_1} + \cdots + \widetilde{F_k}
= \widetilde{F}$. 

\begin{theorem}\label{CIaffinemulti}
Assume that the family $Y=f^{-1}_{1}(0)\cap\dots 
\cap f^{-1}_{k}(0)$ of subvarieties in 
$T_{0}=(\CS)^{n}$ 
is sch\"{o}n. Then for $\lambda\notin R_{f}$ 
the multiplicity of the factor $t- \lambda$ in 
the characteristic polynomial of the 
middle-dimensional monodromy 
\[\Phi_{n-k}:H^{n-k}_{c}(Y_{t};\CC)
\xrightarrow{\sim}H^{n-k}_{c}(Y_{t};\CC) 
\qquad (0<|t|\ll1)\]
is equal to that in 
\[ 
\prod_{\relint F \subset \Int P, \ \dim F=n} 
(t^{m_F}-1)^{ K_F }, 
\]
where we set 
\[
K_F= 
\dsum_{\begin{subarray}{c} 
m_1,\ldots,m_k \geq 1\\ m_1+\cdots +m_k= \dim F 
\end{subarray}}\Vol_{\ZZ}(
\underbrace{\widetilde{F_1},\ldots, \widetilde{F_1}}_{\text{
$m_1$-times}},\ldots, 
\underbrace{\widetilde{F_k},
\ldots, \widetilde{F_k}}_{\text{$m_k$-times}}) 
\]
by using the normalized mixed volumes 
with respect to the lattice $\Aff ( \widetilde{F} ) 
\cap \ZZ^{n+1} \simeq \ZZ^{\dim F}$ in 
$\Aff ( \widetilde{F} ) \simeq \RR^{\dim F}$. 
\end{theorem}

For each subset $J\subset \{1,\dots,k\}$, set $\RR^{J}:
=\{(x_{1},\dots,x_{k})\in\RR^{k}\mid
x_{j}=0~(j\notin J)\}\simeq\RR^{|J|}$ 
and
\begin{align*}
U_{J}:=\Conv(\bigcup_{j\in J}{\{e_{j}\}\times 
\mathrm{UH}_{f_{j}}})\subset \RR^{J}\times \RR^{n+1},
\\P_{J}:=\Conv(\bigcup_{j\in J}{\{e_{j}\}\times P_{j}}) 
\subset \RR^{J}\times \RR^{n},
\end{align*}
where $e_{j}=(0,\dots,0, 
\overset{j}{\hat{1}}, 0, \dots,0)\in 
\RR^{J}$ is the standard vector.
Obviously, the image of $U_{J}$ by the projection 
$p_{J}\colon\RR^{J}\times 
\RR^{n}\times\RR_{s} \to \RR^{J}\times \RR^{n}$ is $P_{J}$.
We write $\tl{p}$, $\tl{U}$ and $\tl{P}$ for 
$p_{\{1,\dots,k\}}$, $U_{\{1,\dots,k\}}$ and 
$P_{\{1,\dots,k\}}$, respectively.
Let $\tl{\nu}\colon\tl{P}\to\RR$ be the function 
defining the bottom part of the boundary
$\partial \tl{U}$ of $\tl{U}$ and $\tl{\mathcal{S}}$ 
the subdivision of $\tl{P}$ by the
lattice polytopes $\tl{p}(\tl{F})\subset\tl{P}~(\tl{F}\prec\tl{U})$.
By the assumption that the dimension of $P$ is $n$, we have 
$\dim{\tl{U}}=n+k$ and 
$\dim{\tl{P}}=n+k-1$.
We obtain the following generalization of Theorem~\ref{st:5.7} to families of 
complete intersection subvarieties of $T_0= (\CS)^n$.
Recall that for $\lambda \in \CC$ we set
\[
\varepsilon ( \lambda ) =\left\{
\begin{array}{ll}
1 &( \lambda =1 )\\\\ 
0 &( \lambda \not= 1 ). \\
\end{array}
\right.
\]

\begin{theorem}\label{propci} 
Let $Y=f^{-1}_{1}(0)\cap\dots\cap f^{-1}_{k}(0)$ be a family of 
subvarieties of $T_{0}=(\CS)^n$.
Assume that $Y$ is sch\"{o}n. 
Then for any $\lambda \in \CC$ the equivariant refined limit mixed Hodge 
polynomial $E_{\lambda}(Y_{\infty};u,v,w)\in\ZZ[u,v,w]$ 
for the eigenvalue $\lambda$
is given by
\begin{align*}
(uvw^{2})^{k}E_{\lambda}(Y_{\infty};u,v,w)= & 
\varepsilon ( \lambda ) \cdot (uvw^2-1)^n 
\\
 + \sum_{\emptyset\neq 
J\subset\{1,\dots,k\}} & 
(-1)^{\dim P_J -1} 
(uvw^2-1)^{n+ |J| -1- \dim P_J} \cdot 
h^{*}_{\lambda}(P_{J},\tl{\nu}|_{P_{J}};u,v,w).
\end{align*}
\end{theorem}
\begin{proof}
We prove the assertion 
only in the case where $\lambda \not= 1$. 
The proofs for the other cases are similar. 
We use the Cayley trick of Danilov-Khovanskii 
\cite{D-K} in its refined form of \cite{E-T}.
For sufficiently small $\varepsilon>0$ we set $B^{*}=
B(0,\varepsilon)^{*}\subset\CC$.
Then we have $Y\subset B^*\times T_{0}$.
Moreover we set
\[\Omega:=\{(t,(x_{1},\dots,x_{n}),[\alpha_{1}:
\dots:\alpha_{k}])\in B^{*}\times T_{0}
\times \PP^{k-1} \ | \ \sum_{i=1}^{k}{\alpha_{i}f_{i}(t,x)\neq 0}\}.\]
Then there exists a projection $\Omega\to (B^*\times T_{0})\setminus Y$
which is a locally trivial fibration with fiber $\CC^{k-1}$.
Hence it follows from Lemma \ref{new-lemma} that for any $\lambda\neq 1$
we have
\begin{align*}
E_{\lambda}(\Omega_{\infty};u,v,w)&=(uvw^{2})^{k-1}E_{\lambda}(
((B^*\times T_{0}) \setminus Y)_{\infty};u,v,w)\\
&=(uvw^{2})^{k-1}(E_{\lambda}((B^*\times T_{0})_{\infty};u,v,w)-
E_{\lambda}(Y_{\infty};u,v,w))\\
&=-(uvw^2)^{k-1}E_{\lambda}(Y_{\infty};u,v,w).
\end{align*}
For each non-empty subset $J\subset\{1,\dots,k\}$
we define a subset $T_{J} \simeq (\CS)^{|J|-1}$ 
of $\PP^{k-1}$ by
\[T_{J}:=\{[\alpha_{1}:\dots:\alpha_{k}]\in\PP^{k-1}\ |
 \ \alpha_{j}\neq0~(j\in J),
\alpha_{j}=0~(j\notin J)\} \simeq (\CS)^{|J|-1}.\]
Moreover we set 
\begin{align*}
\Omega_{J}&:=\Omega\cap(B^{*}\times T_{0} \times T_{J}),\\
Y_{J}&:=(B^{*}\times T_{0}\times T_{J})\setminus \Omega_{J}\\
&=\{(t,(x_{1},\dots,x_{n}),[\alpha_{1}:\dots:\alpha_{k}])
\in B^{*}\times T_{0} \times T_{J} \ | \ \sum_{j\in J}
\alpha_{j}f_{j}(t,x)=0\}.
\end{align*}
Then for $\lambda\neq 1$ we have
\begin{align*}
E_{\lambda}(\Omega_{\infty};u,v,w)&=\sum_{J\neq \emptyset}
E_{\lambda}(\Omega_{J,\infty};u,v,w)\\
&=\sum_{J\neq \emptyset}(E_{\lambda}(
(B^{*}\times T_{0}\times T_{J})_{\infty};u,v,w)
-E_{\lambda}(Y_{J,\infty};u,v,w))\\
&=-\sum_{J\neq \emptyset}E_{\lambda}(Y_{J,\infty};u,v,w).
\end{align*}
We thus obtain the equality
\[(uvw^2)^{k-1}E_{\lambda}(Y_{\infty};u,v,w)=\sum_{J\neq 
\emptyset}E_{\lambda}(Y_{J,\infty};u,v,w).\]
It is easy to check that $Y_{J}$ is sch\"on.
Note that $U_{J}$ is $\mathrm{UH}_{\sum_{j\in J}{\alpha_{j}f_{j}}}$.
Therefore applying Theorem~\ref{st:5.7} to the families 
$Y_{J}\subset B^{*}\times T_{0}
\times T_{J}$
we obtain
\[(uvw^2)E_{\lambda}(Y_{J,\infty};u,v,w)
=(-1)^{\dim P_J -1}
(uvw^2-1)^{n+ |J| -1- \dim P_J} \cdot 
h^{*}_{\lambda}(P_{J},\tl{\nu}|_{P_{J}};u,v,w).\]
Now the assertion follows immediately. 
\end{proof}

\begin{corollary}\label{corpropci} 
In the situation of Theorem \ref{propci}, 
assume also that $\dim P_i=n$ for any 
$1 \leq i \leq k$. 
Then for $\lambda \in \CC$ we have 
\[ 
(uvw^{2})^{k}E_{\lambda}(Y_{\infty};u,v,w)= 
\varepsilon ( \lambda ) \cdot (uvw^2-1)^n 
 + \sum_{\emptyset\neq 
J\subset\{1,\dots,k\}}
(-1)^{n+|J|} 
h^{*}_{\lambda}(P_{J},\tl{\nu}|_{P_{J}};u,v,w).
\] 
\end{corollary}

From now on, we consider only families 
$Y=f^{-1}_{1}(0)\cap\dots\cap f^{-1}_{k}(0)$ of 
subvarieties of $X_{0}= \CC^n$. 
The corresponding results for families of 
subvarieties of $T_{0}=( \CC^*)^n$ 
can be obtained similarly. 
For a cell $F\in\tl{\mathcal{S}}$ let $\tl{\nu_{F}}
\colon \Aff(F)\simeq \RR^{\dim{F}}\longrightarrow \RR$ 
be the (affine) linear extension of $\tl{\nu}|_{F}\colon F\longrightarrow \RR$ and define a positive integer $\tl{m_{F}}$ to be the minimal one $m$ for which $m\cdot\tl{\nu_{F}}$ takes only integer values on $\Aff(F)\cap\ZZ^{n+k}$.
Let $\Delta$ be the convex hull of the points $e_1, \ldots, e_k$ 
in $\RR^k$. Then $\Delta$ is a $(k-1)$-dimensional 
lattice simplex and we have $\tl{P} \subset 
\Delta \times \RR_+^n$. 
Now let us set 
\[
\tl{P}_{\infty}:=\overline{ 
 \{ \Int ( \Delta ) \times \Int(
\RR^{n}_{+}) \} \cap \partial{\tl{P}}}\subset \partial{\tl{P}}.
\]
Then we define a finite subset $\tl{R_{f}}\subset \CC$ by
\[ \tl{R_{f}} =\bigcup_{F\subset \tl{P}_{\infty}}\{\lambda\in
\CC\ |\ \lambda^{ \tl{m_{F}} }=1\}\subset\CC.\]

\begin{lemma}
We have $R_f= \tl{R_{f}}$. 
\end{lemma}

\begin{proof}
We shall say that a face $Q \prec \tl{P}$  
is a side face of $\tl{P}$ if its image by the 
projection $r: \Delta \times \RR_+^n \rightarrow 
\Delta$ is equal to $\Delta$. Note also that 
the inverse image of the barycener of $\Delta$ by 
the map $r|_{\tl{P}}: \tl{P} \rightarrow \Delta$ 
is similar to the Minkowski sum 
$P=P_1+ \cdots +P_k$. Hence there exists a natural 
bijection between the set of the side faces of $\tl{P}$ 
and that of the faces of $P$. Moreover for any 
cell $F\in\tl{\mathcal{S}}$ in $\tl{P}_{\infty}$ 
there exist another cell 
$F^{\prime} \in\tl{\mathcal{S}}$ in $\tl{P}_{\infty}$ 
and a side face $Q \prec \tl{P}$ of $\tl{P}$ 
such that $F \prec F^{\prime}$ and 
$\relint F^{\prime} \subset \relint Q$. 
Then we have $\tl{m_F} | \tl{m_{F^{\prime}}}$. 
This implies that for the definition of 
$\tl{R_f}$ it suffices to consider only 
cells $F\in\tl{\mathcal{S}}$ in $\tl{P}_{\infty}$ 
whose relative interiors are contained in 
those of side faces of $\tl{P}$. In fact, there 
exists also a natural 
bijection between the set of such cells 
$F\in\tl{\mathcal{S}}$ 
and that of the cells $G \in \mathcal{S}$ 
in $P_{\infty}$. For the cell  
$F\in\tl{\mathcal{S}}$ in $\tl{P}_{\infty}$ 
let $F_{\rm red} \in \mathcal{S}$ 
be the corresponding cell in $P_{\infty}$. 
Then it is easy to show that $\tl{m_F} = 
m_{F_{\rm red}}$. We thus obtain the 
equality $R_f= \tl{R_{f}}$. 
\end{proof}

Now we have the following generalization of 
Theorem~\ref{th:syucyuu} to families of
complete intersection subvarieties of $\CC^n$.

\begin{theorem}\label{th:cicon}
Let $Y=f^{-1}_{1}(0)\cap\dots\cap f^{-1}_{k}(0)$ 
be a family of subvarieties of $X_{0}=\CC^n$.
Assume that $Y$ is sch\"{o}n.
Then for any $\lambda\notin R_{f}= 
\tl{R_{f}}$
the equivariant refined limit mixed Hodge polynomial 
$E_{\lambda}(Y_{\infty};u,v,w)\in\ZZ[u,v,w]$ for the eigenvalue $\lambda$
is concentrated in degree $n-k$ in the variable $w$ and given by
\begin{align*}
E_{\lambda}(Y_{\infty};u,v,w)&=(-1)^{n-k}w^{n-k}\sum_{p,q}
h^{p,q}(H^{n-k}_{c}(Y_{\infty};\CC)_{\lambda})u^{p}v^{q}\\
&=(-1)^{n-k}\frac{w^{n-k}}{u^{k}v^{k}}l^{*}_{\lambda}(\tl{P},\tl{\nu};u,v)\\
&=(-1)^{n-k}\frac{w^{n-k}}{u^{k}v^{k}}
\sum_{F\in\tl{\mathcal{S}}}v^{\dim{F}+1}
l^{*}_{\lambda}(F,\tl{\nu}|_F;uv^{-1})
\cdot l_{\tl{P}}(\tl{\mathcal{S}},F;uv).
\end{align*}
In particular, by setting $u=v=s$ and $w=1$ we have
\begin{align*}
E_{\lambda}(Y_{\infty};s,s)&=(-1)^{n-k}\sum_{m\geq 0}(
\sum_{p+q=m}h^{p,q}(H^{n-k}_{c}(Y_{\infty};\CC)_{\lambda}))s^{m}\\
&=(-1)^{n-k}\frac{1}{s^{2k}}\sum_{F\in\tl{\mathcal{S}}}
s^{\dim{F}+1}l^{*}_{\lambda}(F,\tl{\nu}|_{F};1)
\cdot l_{\tl{P}}(\tl{\mathcal{S}},F;s^2).
\end{align*}
\end{theorem}

\begin{proof}
For a possibly empty subset $I\subset \{1,\dots,n\}$,
we define a subset $T^{I}$ of $X_{0}=\CC^n$ by
\[T^{I}:=\{(x_{1},\dots,x_{n})\in X_{0} \ | \ 
x_{i}=0~(i\notin I), x_{i}\neq 0~(i\in I)\}\simeq(\CS)^{|I|}.\]
Then we have a decomposition $X_{0}= \CC^n =\bigsqcup_{
I\subset\{1,\dots,n\}}T^{I}$ of $X_{0}= \CC^n$.
We also define polynomials $f_{j,I}\in \BK[(x_{i})_{i\in I}]$ 
by substituting $0$ into the variable $x_{i}~(i\notin I)$ of $f_{j}$,
a family of subvarieties $Y^{I}$ of $T^{I}$ by 
$Y^{I}:=f^{-1}_{1,I}(0)\cap\dots\cap f^{-1}_{k,I}(0) 
\subset B^* \times T^I$ 
and polytopes $P_{j}^{I}$ in $\RR^{I}$ by $P^{I}_{j}:
=P_{j}\cap\RR^{I}$. We set $P^I= P^{I}_1 + \cdots + 
P^{I}_k= P \cap \RR^{I}$. 
Then by Lemma \ref{new-lemma} we have 
\[
E_{\lambda}(Y_{\infty};u,v,w)= 
\sum_{I\subset \{1,\dots,n\}} E_{\lambda}(Y_{\infty}^{I};u,v,w).
\] 
For each non-empty subset $J\subset \{1,\dots, k\}$ 
we define a polytope $P^{I}_{J}$ in $\RR^{J}\times \RR^{I}$
by $P^{I}_{J}:=\Conv(\bigcup_{j\in J}{\{e_{j}\}
\times P_{j}^{I}})$. 
We shall say that a face $Q \prec \tl{P}$ of $\tl{P}$ 
is relevant if $Q \not\subset \tl{P}_{\infty}$. 
If $Q \prec \tl{P}$ is relevant, then for 
any face $\sigma$ of the polyhedron 
$\Delta \times \RR_+^n$ containing $Q$ 
the face $\sigma \cap \tl{P} \prec \tl{P}$ 
of $\tl{P}$ is also relevant. 
Moreover there exist a possibly empty 
subset $I \subset \{1,\dots, n\}$ and 
a non-empty one $J\subset \{1,\dots, k\}$ 
such that $Q=P_J^I$ and $\dim P_J^I = 
|I|+|J|-1$. 
For each $I\subset \{1,\dots,n\}$ denote 
by $S^I$ the set consisting of 
non-empty subsets $J\subset\{1,\dots,k\}$ 
such that $P_J^I$ is relevant. 
Then by Theorem \ref{propci} 
for $\lambda\notin R_{f}= \tl{R_{f}}$ we have 
\[ 
(uvw^{2})^{k}E_{\lambda}(Y^I_{\infty};u,v,w)= 
\sum_{J \in S^I}
(-1)^{|I|+|J|} 
h^{*}_{\lambda}(P^I_J, 
\tl{\nu}|_{P^I_J};u,v,w).
\] 
Moreover for each relevant face 
$P^I_J \prec \tl{P}$ of $\tl{P}$ 
by Definition~\ref{def:poly} we have
\begin{align}\label{cieq3}
h^{*}_{\lambda}(P^{I}_{J},\tl{\nu}|_{P^{I}_{J}};u,v,w)
=\sum_{Q\prec P^{I}_{J}}
w^{\dim Q+1}l^{*}_{\lambda}(Q,\tl{\nu}|_{Q};u,v)
\cdot g([Q,P^{I}_{J}];uvw^2).
\end{align}
If $Q\prec P^{I}_{J}$ is not a 
relevant face of $\tl{P}$, then $Q \subset \tl{P}_{\infty}$ 
for $\lambda\notin R_{f}= \tl{R_{f}}$ 
we have 
$l^{*}_{\lambda}(Q,\tl{\nu}|_{Q};u,v)=0$. Note also 
that $l^{*}_{\lambda}(\emptyset,\tl{\nu}|_{\emptyset};u,v)=0$ 
for any $\lambda\neq1$. 
Thus we may assume $Q$ is not empty set in the right 
hand side of the equation~(\ref{cieq3}). 
Hence as in the proof of Theorem \ref{th:syucyuu}, 
most of terms in the 
calculation of $(uvw^2)^{k}E_{\lambda}(Y_{\infty};u,v,w)$ 
cancel each other. 
Eventually, we obtain the desired formula 
\[(uvw^2)^{k}E_{\lambda}(Y_{\infty};u,v,w)=(-1)^{n+k}
w^{n+k}l^{*}_{\lambda}(\tl{P};\tl{\nu},u,v).\]
\end{proof}

As in Corollary \ref {Ncc}, 
by Theorems \ref{th:555} and \ref{th:cicon} 
we obtain the following result. 

\begin{corollary}\label{CCNcc}
In the situation of Theorem \ref{th:cicon}, 
for any $\lambda\notin R_{f}$ 
and $t\in \CS$ such that $0<|t|\ll1$ we have the 
concentration 
\[H^{j}_{c}(Y_{t};\CC)_{\lambda} 
\simeq 0 \qquad (j \not= n-k)\] 
and the filtration on the only non-trivial 
cohomology group $H^{n-k}_c(Y_{t};\CC)_{\lambda}$ 
induced by Deligne's weight filtration on 
$H^{n-k}_c(Y_{t};\CC)$ is concentrated in 
degree $n-k$. 
\end{corollary}

By Theorem~\ref{th:cicon} 
and Corollary \ref{CCNcc}, for any $\lambda\notin 
R_{f}$ we can describe the
Jordan normal form of the 
middle-dimensional monodromy 
\[\Phi_{n-k}:H^{n-k}_{c}(Y_{t};\CC)_{\lambda}
\xrightarrow{\sim}H^{n-k}_{c}(Y_{t};\CC)_{\lambda}\]
as in Theorem~\ref{torusjordan}. 
Recall that the dimension of $\tl{P}$ is $n+k-1$,
and hence for a cell $F\in\tl{\mathcal{S}}$ the 
local $h$-polynomial $l_{\tl{P}}(\tl{\mathcal{S}},
F;t)\in\ZZ[t]$ has non-negative coefficients and the symmetry
\[l_{\tl{P}}(\tl{\mathcal{S}},F;t)=t^{n+k-1
-\dim{F}}l_{\tl{P}}(\tl{\mathcal{S}},F;t^{-1}).\]
Moreover it is unimodal.
Hence there exist non-negative integers 
$l_{F,i}~(0\leq i\leq\lfloor\frac{n+k-1-\dim{F}}{2}\rfloor)$ such that
\begin{align*}
l_{\tl{P}}(\tl{\mathcal{S}},F;t)&=l_{F,0}(1+t+
t^2+\dots+t^{n+k-1-\dim{F}})\\
&+l_{F,1}(t+t^2+\dots+t^{n+k-1-\dim{F}-1})\\
&+l_{F,2}(t^2+\dots+t^{n+k-1-\dim{F}-2})\\
&+\cdots\cdots.
\end{align*}
We set
\[\tl{l}_{\tl{P}}(\tl{\mathcal{S}},F;t)=
\sum_{i=0}^{\lfloor{\frac{n+k-1-\dim{F}}{2}}\rfloor}l_{F,i}t^i.\]

\begin{theorem}\label{JBFCI} 
Let $Y=f^{-1}_{1}(0)\cap\dots\cap f^{-1}_{k}(0)$ 
be a family of subvarieties of $X_{0}=\CC^n$.
Assume that the family $Y$ is sch\"on.
For $\lambda\in\CC$ and $m\geq 1$ denote by $J_{\lambda,m}$
the number of the Jordan blocks in the monodromy automorphism
\[\Phi_{n-k}:H^{n-k}_{c}(Y_{t};\CC)\xrightarrow{\sim}
H^{n-k}_{c}(Y_{t};\CC)\qquad(0<|t|\ll1)\]
for the eigenvalue $\lambda$ with size $m$.
Then for $\lambda\notin R_{f}$ we have
\[\sum_{m=0}^{n-k}J_{\lambda,n-k+1-m}s^{m+2k}=
\sum_{F\in\tl{\mathcal{S}}}s^{\dim{F}+1}l^{*}_{\lambda}(
F,\tl{\nu}|_{F};1) \cdot 
\tl{l}_{\tl{P}}(\tl{\mathcal{S}},F;s^2).\]
\end{theorem}

The multiplicities of the eigenvalues 
$\lambda\notin R_{f}$ in the monodromy 
$\Phi_{n-k}$ are described more simply as follows. 
For a cell $F \in \mathcal{S}$ 
let $Q_F \prec P$ be the unique 
face of $P$ such that $\relint F \subset \relint Q_F$. 

\begin{theorem}\label{CCCIemulti}
In the situation of Theorem \ref{JBFCI}, 
for $\lambda\notin R_{f}$ 
the multiplicity of the factor $t- \lambda$ in 
the characteristic polynomial of the 
middle-dimensional monodromy 
\[\Phi_{n-k}:H^{n-k}_{c}(Y_{t};\CC)
\xrightarrow{\sim}H^{n-k}_{c}(Y_{t};\CC) 
\qquad (0<|t|\ll1)\]
is equal to that in
\[
\prod_{F \not\subset P_{\infty}, \ \dim F= \dim Q_F} 
(t^{m_F}-1)^{(-1)^{n- \dim F} K_F }, 
\]
where we define the integers $K_F$ as in 
Theorem \ref{CIaffinemulti}. 
\end{theorem}

\bibliographystyle{plain}
\bibliography{referenceforonthemon}

\end{document}